\documentclass{amsart}
\usepackage{amsmath,amssymb}

\allowdisplaybreaks[2]

\newtheorem{thm}{Theorem}[section]
\newtheorem{lemma}[thm]{Lemma}
\newtheorem{theorem}[thm]{Theorem}
\newtheorem{proposition}[thm]{Proposition}
\newtheorem{corollary}[thm]{Corollary}

\newtheorem{notation}[thm]{Notation}
\newtheorem{definition}[thm]{Definition}

\newtheorem{remark}[thm]{Remark}
\newtheorem*{thmintro}{Theorem \ref{thm:Neumann product}}
\newtheorem*{thmintro2}{Theorems \ref{thm:int} and \ref{thm:rank int}}
\newtheorem*{corintro}{Corollary \ref{cor:SHNC}}

\newcommand{\RR}{\mathbb R}
\newcommand{\RRR}{\mathbb {R}_{\geq 0}}
\newcommand{\R}{\mathcal R}
\newcommand{\N}{\mathcal N}
\newcommand{\C}{\mathfrak C}

\newcommand{\gd}{\delta}
\newcommand{\gD}{\Delta}

\newcommand{\gG}{\Gamma}
\newcommand{\id}{\mathrm{id}}
\newcommand{\SCyl}{\mathcal{S} \mathrm{Cyl}}
\newcommand{\SCurr}{\mathcal{S} \mathrm{Curr}}
\newcommand{\rk}{\overline{\mathrm{rk}}}

\newcommand{\tn}[1]{\textnormal{#1}}
\newcommand{\ti}[1]{\textit{#1}}
\newcommand{\Sub}{\mathrm{Sub}}
\newcommand{\SSub}{\widehat{\mathrm{Sub}}}
\renewcommand{\:}{\colon}
\renewcommand{\ne}{\not= \emptyset}
\renewcommand{\tilde}{\widetilde}

\begin{document}

\title[An intersection functional on the space of subset currents]{An intersection functional on the space of subset currents on a free group}

\author[D. Sasaki]{Dounnu Sasaki}

\address{Faculty~of~Science~and~Engineering, Waseda~University, Okubo~3-4-1, Shinjuku, Tokyo~169-8555, Japan}

\email{dounnu-daigaku@moegi.waseda.jp}



\begin{abstract}
Kapovich and Nagnibeda introduced the space $\mathcal{S} {\rm Curr}(F_N)$ of subset currents on a free group $F_N$ of rank $N\geq 2$, which can be thought of as a measure-theoretic completion of the set of all conjugacy classes of finitely generated subgroups of $F_N$. We define a product $\mathcal{N} (H,K)$ of two finitely generated subgroups $H$ and $K$ of $F_N$ by the sum of the reduced rank $\overline{\rm {rk}}(H\cap gKg^{-1})$ over all double cosets $HgK\ (g\in F_N)$, and extend the product $\mathcal{N}$ to a continuous symmetric $\mathbb{R}_{\geq 0}$-bilinear functional $\mathcal{N} \colon \mathcal{S} \textnormal{Curr} (F_N)\times \mathcal{S} \textnormal{Curr} (F_N)\to \mathbb {R}_{\geq 0}$. We also give an answer to a question presented by Kapovich and Nagnibeda.
The definition of $\mathcal{N}$ originates in the Strengthened Hanna Neumann Conjecture, which has been proven independently by Friedman and Mineyev, and can be stated as follows: for any finitely generated subgroups $H$ and $K$ of $F_N$ the inequality $\mathcal{N} (H,K)\leq \overline{\rm rk} (H) \overline{\rm rk} (K)$ holds. As a corollary to our theorem, this inequality is generalized to the inequality for subset currents.
\end{abstract}
\maketitle

\tableofcontents

\section{Introduction}
In \cite{KN13} Kapovich and Nagnibeda introduced the space $\SCurr (F_N)$ of \ti{subset currents} on a free group $F_N$ of rank $N\geq 2$ as an analogy of the space of geodesic currents on $F_N$ and on hyperbolic surfaces.
Geodesic currents on a hyperbolic surface, which were introduced by Bonahon \cite{Bon86,Bon88}, have been used successfully in the study of mapping class group and Teichm\"uller space of hyperbolic surfaces. 
Especially, the (geometric) intersection number, which is an $\RRR$-bilinear continuous functional on the space of geodesic currents, has played an essential role in the study.

Bonahon also introduced the notion of geodesic currents on hyperbolic groups in \cite{Bon88b}. 
Formally, a \ti{geodesic current} on a hyperbolic group $G$ is a Borel measure on $\partial ^2G:=\{(x,y)\in \partial G\ |\ x\not=y \}$  that is locally finite (i.e., finite on compact subsets), $G$-invariant and invariant with respect to the ``flip map'' $\partial ^2G \to \partial ^2G ; (x,y)\mapsto (y,x)$. 
Equivalently, a geodesic current on $G$ is a locally finite $G$-invariant Borel measure on the space of $2$-element subsets of $\partial G$.
The space of geodesic currents on $G$ can be thought of as a measure-theoretic completion of the set of all conjugacy classes in $G$, which correspond to free homotopy classes of closed curves in a hyperbolic surface when $G$ is the fundamental group of the surface.

Geodesic currents on $F_N$ were first studied by Reiner Martin in his 1995 PhD thesis \cite{Mar95}. They have been used analogically in the study of the outer automorphism group and Outer space of $F_N$, which was introduced by Culler and Vogtmann \cite{CV86} as a free-group analogy of Teichm\"uller space. See \cite{Kap05,Kap06,KN07} for some fundamental results of geodesic currents on free groups.
Recently,  geodesic currents on $F_N$ were used to describe the boundary of the complex of free factors for $F_N$ in \cite{BR12,Ham12}
and also used to investigate the Poisson boundary of $\tn{Out}(F_N)$ in \cite{Hor14}. A key tool for these works is the geometric intersection form between the closure of the unprojectivized Outer space and the space of geodesic currents on $F_N$ (see \cite{KL09}), which is an analogy of the intersection number on the space of geodesic currents on a hyperbolic surface.
The complex of free factors for $F_N$ is an analogue for the complex of curves for a hyperbolic surface and has been proven to be a Gromov hyperbolic space in \cite{BF14}. 

Kapovich and Nagnibeda \cite{KN13} generalized the notion of geodesic currents and introduced subset currents on $F_N$, and, more generally, on hyperbolic groups. In \cite{KN13} much theory of geodesic currents on $F_N$ was extended to the theory of subset currents on $F_N$, such as the co-volume form between Outer space and the space of subset currents on $F_N$, which is a generalization of the geometric intersection form. A relation between theory of invariant random subgroups and subset currents on $F_N$ is also indicated in the introduction in \cite{KN13}.

\smallskip

The purpose of this paper is to construct a ``natural'' intersection functional $\N$ on the space $\SCurr(F_N)$ of subset currents on $F_N$.

A \ti{subset current} on $F_N$ is a positive $F_N$-invariant locally finite Borel measure on the space $\C_N$ of all closed subsets of the hyperbolic boundary $\partial F_N$ consisting of at least two points, where we endow $\C_N$ with the subspace topology of the Vietoris topology on the hyperspace of $\partial F_N$. 
The space $\SCurr (F_N)$ of subset currents on $F_N$ is equipped with the weak-* topology and with the $\RRR$-linear structure.
By the definition, the space of geodesic currents on $F_N$ is canonically embedded in $\SCurr (F_N)$ as a closed subspace.
Kapovich and Nagnibeda defined the \ti{counting subset current} $\eta_H\in \SCurr (F_N)$ for a nontrivial finitely generated subgroup $H\leq F_N$ and proved that
\[ \{ c \eta_H \ |\ c\geq 0, H\ \tn{is a nontirivial finitely generated subgroup of }F_N \}\]
is a dense subset of $\SCurr (F_N)$. As a matter of convenience, for the trivial subgroup $H=\{ \id \}$ of $F_N$ we set $\eta_H=0\in \SCurr (F_N)$. 
Counting subset currents have the following properties.
For a finitely generated subgroup $H\leq F_N$ and a finite index subgroup $H'$ of $H$, we have $\eta_{H'}=[H:H']\eta_{H}$. If two finitely generated subgroups $H,H'\leq F_N$ are conjugate, then $\eta_{H}=\eta_{H'}$.
From the above properties of counting subset currents, we say that $\SCurr (F_N)$ can be thought of as a measure-theoretic completion of the set of all conjugacy classes of finitely generated subgroups of $F_N$. 

The action of the automorphism group $\operatorname{Aut}(F_N)$ of $F_N$ on $\partial F_N$ induces a continuous and $\RRR$-linear action on $\SCurr (F_N)$, and for $\varphi \in \tn{Aut}(F_N)$ and a finitely generated subgroup $H\leq F_N$ we have $\varphi \eta_{H}=\eta_{\varphi (H)}$. Since subset currents are $F_N$-invariant, the action of $\tn{Aut}(F_N)$ factors through the action of the outer automorphism group $\tn{Out}(F_N)$ on $\SCurr (F_N)$.

For a finitely generated free group $F$, the \ti{reduced rank} $\rk (F)$ is defined as
\[ \rk (F) :=\tn{max}\{\tn {rank}(F)-1,0\} ,\]
where $\tn{rank}(F)$ is the cardinality of a free basis of $F$. Let $\gD$ be a finite connected graph whose fundamental group is isomorphic to $F$. Then we have 
\[ \rk (F)=\tn{max} \{ -\chi (\gD ) ,0\} ,\]
where $\chi (\gD )$ is the Euler characteristic of $\gD $.
Note that for a finite index subgroup $H\leq F$ we can obtain an equation 
\[ \rk (H) = [F:H]\, \rk (F) ,\]
which follows from a covering space argument.
Kapovich and Nagnibeda \cite{KN13} proved that there exists a unique continuous $\RRR$-linear functional
 \[ \rk \: \SCurr (F_N)\to \RRR \]
such that for every finitely generated subgroup $H\leq F_N$ we have
\[ \rk (\eta _H)=\rk (H).\]
Moreover, $\rk$ is $\tn{Out}(F_N)$-invariant.
The map $\rk $ is called the \ti{reduced rank functional}.

Define a product $\N (H,K)$ of two finitely generated subgroups $H,K\leq F_N $ as
\[ \N (H,K) :=\sum_{HgK\in H\backslash F_N /K } \rk (H\cap gKg^{-1}),\]
where $H\backslash F_N/K$ is the set of all double cosets $HgK\ (g\in F_N)$. Note that $H\cap gKg^{-1}\not=\{ \id \}$ for only finitely many double cosets $HgK$. This definition originates in the Strengthened Hanna Neumann Conjecture, which has been proven independently by Friedman \cite{Fri11} and Mineyev \cite{Min12a}. By using the product $\N$ it can be stated as follows:
\[ \N (H,K)\leq \rk (H) \rk (K) \]
is satisfied for any finitely generated subgroups $H,K\leq F_N$.
The product $\N(H,K)$ is closely related to a \ti{fiber product graph} corresponding to $H$ and $K$, that is, each non-zero term of the sum in $\N(H,K)$ is corresponding to a non-contractible connected component of the fiber product graph (see Section \ref{sec:product}). Using the description by the fiber product graph, we can easily see that $\N$ has the following property: if $H'$ and $K'$ are finite index subgroups of $H$ and $K$ respectively, then we have
\[ \N (H', K')=[H:H'][K:K']\N (H,K).\]
Therefore it is natural to ask whether $\N$ extends to a continuous $\RRR$-bilinear functional on $\SCurr (F_N)$.

From the $\RRR$-linearity of the reduced rank functional, we have
\begin{align*}
\N (H,K)&=\sum_{HgK\in H\backslash F_N /K } \rk (H\cap gKg^{-1}) \\
	&=\rk \left( \sum_{HgK\in H\backslash F_N/K} \eta_{H\cap gKg^{-1}}\right).
\end{align*}
Kapovich and Nagnibeda \cite{KN13} asked whether there exists a continuous $\RRR$-bilinear map
\[ \pitchfork \: \SCurr (F_N)\times \SCurr (F_N)\to \SCurr (F_N)\]
such that for any finitely generated subgroups $H,K\leq F_N$ we have
\[ \pitchfork (\eta _H, \eta _K )=\sum_{HgK\in H\backslash F_N/K} \eta_{H\cap gKg^{-1}}.\]
If such a map $\pitchfork$ exists, then we immediately see that the product $\N$ is extended to a continuous $\RRR$-bilinear functional $\N \:\SCurr (F_N)\times \SCurr (F_N)\to \RRR$. Moreover, the map $\pitchfork$ can be considered as a measure theoretical generalization of the construction of the fiber product graph.
However, we prove that the map $\pitchfork$ can not be continuous due to the requirements on $\pitchfork$ (see Proposition \ref{prop:not conti}).
Nevertheless, we can establish the following theorem by a different approach.
\begin{thmintro}
There exists a unique continuous symmetric $\RRR$-bilinear functional
\[ \N \: \SCurr (F_N)\times \SCurr (F_N)\rightarrow \RRR\]
such that for any finitely generated subgroups $H,K\leq F_N$ we have
\[ \N (\eta_H ,\eta_K )= \N (H,K).\]
Moreover, $\N$ is $\tn{Out}(F_N)$-invariant.
\end{thmintro}
We call $\N$ the \ti{intersection functional}. Our strategy of proving Theorem \ref{thm:Neumann product} is based on the proof of the existence of the reduced rank functional $\rk$. We construct $\N$ by a careful usage of \ti{occurrences}, which were introduced in \cite{KN13}.

As a corollary to our theorem,
 the inequality in the Strengthened Hanna Neumann Conjecture is generalized to the inequality for subset currents:
\begin{corintro}
Let $\mu ,\nu \in \SCurr (F_N)$. The following inequality holds:
\[ \N (\mu ,\nu )\leq \rk (\mu )\rk (\nu ).\]
\end{corintro}
Since $\N (F_N, H)=\rk (H)$ for every finitely generated subgroup $H\leq F_N$, the intersection functional $\N $ is an extension of the reduced rank functional $\rk$.

Inspired by the question presented by Kapovich and Nagnibeda, we also prove the following theorems.
\begin{thmintro2}
Let $I$ be the intersection map:
\[ \C_N \times \C_N \rightarrow \{ \tn{closed subsets of}\ \partial F_N\} ;\ (S_1,S_2)\mapsto S_1\cap S_2.\]
For $\mu, \nu \in \SCurr (F_N)$ we can obtain a subset current $\widehat{I}(\mu ,\nu )$ by defining
\[ \widehat{I}(\mu ,\nu )(U):=\mu \times \nu (I^{-1}(U))\] 
for any Borel subset $U\subset \C_N$.
Then the map $\widehat{I}$ is a non-continuous $\RRR$-bilinear map
\[ \widehat{I}\: \SCurr (F_N)\times \SCurr (F_N)\to \SCurr (F_N)\]
satisfying the conditions that 
\[ \widehat{I}(\eta_H, \eta_K )=\sum_{HgK\in H\backslash F_N/K} \eta_{H\cap gKg^{-1}}\]
for any finitely generated subgroups $H,K\leq F_N$, and that
\[ \rk \circ \widehat{I}=\N.\]
\end{thmintro2}
\smallskip
\noindent \textbf{Organization of this paper.} 
In Section 2 we set up notation and summarize without proofs some important properties on subset currents in \cite{KN13}. We also recall some tools and methods used in the proof of the existence of the reduced rank functional $\rk$.
In Section 3, first, we give an answer to the question posed in \cite{KN13}, and construct the intersection functional $\N$. 
In Section 4 we represent $\N$ by using the intersection map $I$ and the reduced rank functional $\rk$. 

\section{Preliminaries}\label{sec:Preliminaries}
\subsection{Conventions regarding graphs and free groups}
This subsection is based on \cite[Subsections 2.1 and 2.2]{KN13}.
A \ti{graph} is a $0$ or $1$-dimensional CW complex. The set of 0-cells of a graph $\gD$ is denoted by $V(\gD )$ and its elements are called \ti{vertices} of $\gD$. The set of (closed) 1-cells of a graph $\gD$ is denoted by $E_{\tn{top}}(\gD )$ and its elements are called \ti{topological edges}. The interior of every topological edge is homeomorphic to the interval $(0,1)\subset \RR $ and thus admits exactly two orientations. A topological edge endowed with an orientation on its interior is called an \ti{oriented edge} of $\gD$.
The set of all oriented edges of $\gD$ is denoted by $E(\gD )$. For every oriented edge $e$ of $\gD$ there are naturally defined (and not necessarily distinct) vertices $o(e)\in V(\gD ) $, called the \ti{origin} of $e$, and $t(e)\in V(\gD )$, called the \ti{terminal} of $e$. Then the boundary of $e$ is $\{ o(e),t(e)\}$.
For an oriented edge $e\in E(\gD)$ changing its orientation to the opposite one produces another oriented edge of $\gD$ denoted by $e^{-1}$ and called the \ti{inverse} of $e$. For a graph $\gD$ giving an \ti{orientation} to $\gD$ is fixing an orientation of every topological edge of $\gD$. 

Let $\gD$ be a graph and $v\in V(\gD)$. The degree of $v$ in $\gD$ is the number of oriented edges whose origin is $v$.

For a graph $\gD$, we always give the path metric $d_{\gD}$ to $\gD$, where the length of each topological edge of $\gD$ is $1$.

A \ti{graph morphism} $f:\gD \rightarrow \gD' $ is a continuous map from a graph $\gD$ to a graph $\gD'$ that maps the vertices of $\gD$ to vertices of $\gD'$ and such that each interior of an edge of $\gD$ is mapped isometrically to an interior of an edge of $\gD'$.
A \ti{graph isomorphism} is a bijective graph morphism.

Let $N\geq 2$ be an integer. We fix a free basis $A=\{ a_1,\dots , a_N\}$ of the free group $F_N$. We denote by $\id$ the identity element of $F_N$. We denote by $\Sub (F_N)$ the set of all non-trivial finitely generated subgroups of $F_N$.

Let $X$ be the Cayley graph of $(F_N, A)$, where 
$V(X):= F_N$, $E_{\tn{top}}(X):=F_N\times A$, and for every topological edge $(g,a)\in E_{\tn{top}}(X)$ the boundary is $\{ g, ga\}$. The Cayley graph $X$ is a tree. The free group $F_N$ naturally acts on $X$ from the left by graph isomorphisms. Let $R_N$ be the quotient graph $F_N\backslash X$ and denote by $q:X\to R_N$ the canonical projection, which is a universal covering map.
The quotient graph $R_N$ consists of one vertex $x_0$ and $N$ loop-edges at the vertex $x_0$, each of which is corresponding to an element of $A$. The graph $R_N$ is called $N$-\ti{rose}. We identify $F_N$ with the fundamental group $\pi _1 (R_N,x_0)$ in a natural way, and also identify the hyperbolic boundary $\partial F_N$ of $F_N$ with the (geodesic) boundary $\partial X$ of $X$. We deal only with $\partial X$ in this paper.
However, we have to care about whether each statement depends on the basis $A$ or not since $\partial F_N$ exists without fixing any free basis $A$.

\subsection{The space of subset currents on $F_N$}\label{subsec:C_N}
In this subsection, we summarize necessary definitions and properties of subset currents on $F_N$ (see \cite{KN13} for details).

We denote by $\C_N$ the set of all closed subsets $S\subset \partial X$ such that the cardinality $\# S\geq 2$. We endow $\C_N$ with the subspace topology from the Vietoris topology on the hyperspace of $\partial X$, which consists of all closed subsets of $\partial X$. Then, $\C_N$ is a locally compact, second countable, totally disconnected and metrizable space. If we give a distance on $\partial X$ which is compatible with the topology on $\partial X$, then the topology on $\C_N$ induced by the Hausdorff distance coincides with the topology which we defined above.

For an oriented edge $e\in E(X)$, we define the \ti{cylinder} $\tn{Cyl} (e)$ to be the subset of $\partial X$ consisting of equivalence classes of all geodesic rays in $X$ emanating from the oriented edge $e$. A cylinder $\tn{Cyl} (e)$ is an open and compact subset of $\partial X$ for any $e\in E(X)$, and the collection of all $\tn{Cyl} (e) \ (e\in E(X))$ is a basis of $\partial X$.

We denote by $\Sub (X)$ the set of all non-degenerate finite subtrees of $X$, which are finite subtrees with at least two distinct vertices

Let $T\in \Sub (X)$ and let $e_1,\dots ,e_m$ be all the terminal edges of $T$, which are oriented edges whose terminal vertices are precisely the vertices of $T$ of degree $1$. Define the
\ti{subset cylinder} $\SCyl (T)$ to be the subset of $\C_N$ consisting of $S\in \C_N$ satisfying the condition that
\[ S\subset \bigcup_{i=1}^m \tn{Cyl}(e_i)\ \tn{and}\ S\cap \tn{Cyl}(e_i)\not= \emptyset \ (\forall i=1,2,\dots m).\]
For $T\in \Sub (X)$ the subset $\SCyl (T)\subset \C_N$ is compact and open, and the collection of all $\SCyl (T)\ (T\in \Sub (X))$ forms a basis for the topology on $\C_N $.

\smallskip

Note that the left continuous action of $F_N$ on $\partial X$ naturally extends to a left continuous action on $\C_N$.

A \ti{subset current} on $F_N$ is a Borel measure on $\C_N$ that is $F_N$-invariant and locally finite (i.e., finite on all compact subsets of $\C_N$.)

The set of all subset currents on $F_N$ is denoted by $\SCurr (F_N)$. The space $\SCurr (F_N)$ has the $\RRR$-linear structure, and the space $\SCurr (F_N)$ is endowed with the natural weak-$*$ topology, which is characterized by the following proposition (1).

\begin{proposition}[See \cite{KN13}, Proposition 3.7]\label{prop:Scyl is conti} \mbox{}
\begin{enumerate}
\item Let $\mu, \mu_n\in \SCurr (F_N)\ (n=1,2,\dots )$. Then $\lim_{n\to \infty }\mu_n =\mu$ in $\SCurr (F_N)$ if and only if for every $T\in \Sub (X)$ we have 
\[ \lim_{n\to \infty }\mu_n (\SCyl (T))=\mu (\SCyl (T)).\]

\item For each $T\in \Sub (X)$ the functional
\[ \SCurr (F_N)\to \RRR ;\ \mu \mapsto \mu (\SCyl (T))\]
is continuous and $\RRR$-linear.
\end{enumerate}
\end{proposition}

Recall that for a subgroup $H$ of a group $G$ the \ti{commensurator} or \ti{virtual normalizer} $\tn{Comm}_{G}(H)$ of $H$ in $G$ is defined as
\[ \tn{Comm}_{G}(H):= \{ g\in G\ |\ [H:H\cap gHg^{-1}]<\infty \ \tn{and}\ [gHg^{-1}:H\cap gHg^{-1}]<\infty \}.\]
Let $H\in \Sub (F_N)$. The \ti{limit set} $\Lambda(H)$ of $H$ in $\partial X (=\partial F_N)$ is the set of all $\xi \in \partial X$ such that there exists a sequence of $h_n\in H\ (n=1,2,\dots )$ satisfying $\lim_{n\to \infty} h_n=\xi $ in $X\cup \partial X$.
See \cite[Proposition 4.1 and Proposition 4.2]{KN13} for elementary properties of limit sets.

Note that for every $H\in \Sub (F_N)$ we have $\Lambda (H)\in \C_N$, and 
\[ \tn{Comm}_{F_N}(H)=\tn{Stab}_{F_N}(\Lambda (H)),\]
where $\tn{Stab}_{F_N}(\Lambda (H)):=\{ g\in F_N\ |\ g\Lambda (H)=\Lambda (H)\}$ is the stabilizer of $\Lambda(H)$.

For $H\in \Sub (F_N)$ a subset current $\eta_H$, which is called a \ti{counting subset current}, is defined as follows.
Set $\widehat{H}=\tn{Comm}_{F_N}(H)$.

Suppose first that $H=\widehat{H}$. Define a Borel measure $\eta_H$ on $\C_N$ as 
\[ \eta_H:=\sum_{H' \in [H] } \delta _{\Lambda (H')},\]
where $[H]$ is the conjugacy class of $H$ in $F_N$ and $\delta _{\Lambda (H')}$ is the Dirac measure on $\C_N$ which means that for a Borel subset $U\subset \C_N$, if $\Lambda (H')\in U$, then $\delta _{\Lambda (H')}(U)=1$; if $\Lambda (H')\not\in U$, then $\delta _{\Lambda (H')}(U)=0$.

Now let $H$ be an arbitrary nontrivial finitely generated subgroup of $F_N$. Define $\eta_H$ as $\eta_H:=[\widehat{H}:H]  \eta_{\widehat{H}}$. Note that $\tn{Comm}_{F_N}(\widehat{H})=\widehat{H}$ and $[\widehat{H}:H]$ is finite.

We can see that $\eta_H$ is a subset current, especially, locally finite (see \cite[Lemma 4.4]{KN13}).
For the trivial subgroup $H=\{ \id \}$ we set $\eta_H=0\in \SCurr (F_N)$.
A subset current $\mu \in \SCurr (F_N)$ is said to be \ti{rational} if $\mu= r\eta_H $ for some $r\geq 0$ and $H\in \Sub (F_N)$. The set of all rational subset currents is a dense subset of $\SCurr (F_N)$ (see \cite[Theorem 5.8]{KN13}, and \cite{Kap13}).

We observe the following proposition.
\begin{proposition}\label{gH}
Let $H\in \Sub(F_N)$. Then we have 
\[ \eta_H = \sum _{gH\in F_N/H} \delta _{g\Lambda (H)}.\]
\end{proposition}
\begin{proof}
First, we assume $H=\textnormal{Comm}_{F_N}(H)(=\tn{Stab}_{F_N}(\Lambda (H)))$. Then we have the following map:
\[ \Phi : F_N/H \rightarrow [H];\ gH \mapsto gHg^{-1}=\tn{Stab}_{F_N}(g\Lambda (H)).\]
The map $\Phi$ is clearly surjective and we can prove that $\Phi$ is injective as follows.
Let $g_1H, g_2 H \in F_N/H$, and suppose $g_1 Hg_1^{-1}=g_2 H g_2^{-1}$. Then we have 
$g_2^{-1}g_1Hg_1^{-1}g_2=H$, and so $g_2^{-1}g_1\in H$ by the definition of the commensurator. Hence $g_1H=g_2H$. Therefore $\Phi $ is bijective, and this implies
\[ \eta _H=\sum _{H' \in [H]} \delta _{\Lambda (H')}=\sum _{gH\in F_N/H} \delta _{\Lambda (\Phi (gH))}=\sum _{gH\in F_N/H} \delta _{g\Lambda (H)}.\]

In general, put $\widehat{H}:=\tn{Comm}_{F_N}H$ and $m:= [\widehat{H}:H]$. Then we can choose $h_1,\dots ,h_m\in \widehat{H}$ such that
$\{ h_iH\}_{i=1,\dots ,m}$ is a complete system of representatives of $\widehat{H}/H$. If $ \{ g_j \widehat{H}\} _{j\in J}$ is a complete system of representatives of $F_N/\widehat{H}$, then $\{ g_j h_iH\}_{i=1,\dots ,m, j\in J}$ is a complete system of representatives of $F_N/H$.
Since $h_i \in \widehat{H}=\tn{Stab}_{F_N}(\Lambda (H))$ and $\Lambda(H)=\Lambda(\widehat{H})$, we have 
\begin{align*}
\sum _{gH\in F_N/H} \delta _{g\Lambda (H)}=\sum _{i,j} \delta _{g_jh_i\Lambda (H)}
	=&m\sum _{j\in J} \delta _{g_j\Lambda (\widehat{H})}\\
=&m\sum _{g\widehat{H}\in F_N/\widehat{H}} \delta _{g\Lambda (\widehat{H})}
	=m\eta_{\widehat{H}}=\eta _H,
\end{align*}
as required.\end{proof}

If $\varphi \in \tn{Aut }(F_N)$ is an automorphism of $F_N$, then $\varphi$ induces a quasi-isometry of $X$, and moreover, the quasi-isometry extends to a homeomorphism $\varphi \: \partial X \to \partial X$, where we still denote it by $\varphi$. Thus $\tn{Aut}(F_N)$ has a natural action on $\C_N$. Moreover, $\tn{Aut}(F_N)$ acts on $\SCurr (F_N)$ $\RRR$-linearly and continuously by pushing forward. Explicitly, 
\[ (\varphi \mu )(U):=\mu (\varphi ^{-1}(U))\]
for $\varphi \in \tn{Aut}(F_N), \mu \in \SCurr (F_N)$ and every Borel subset $U\subset \C_N$. Then for $\varphi \in \tn{Aut}(F_N)$ and $H\in \Sub (F_N)$ we have $\varphi \eta_H=\eta_{\varphi (H)}$. Since subset currents are $F_N$-invariant, the action of $\tn{Aut}(F_N)$ on $\SCurr (F_N)$ factors through the action of the outer automorphism group  $\tn{Out} (F_N)$ on $\SCurr (F_N)$. The action of $\tn{Out} (F_N)$ on $\SCurr (F_N)$ is effective.

\subsection{$R_N$-graphs and occurrences}

In this subsection following \cite[Subsection 4.2]{KN13} we first define $R_N$-\ti{graphs}, and also \ti{occurreces} for an $R_N$-graph and for $T\in \Sub (X)$, the set of all non-degenerate finite subtrees of $X$. Occurrences play an essential role in studying rational subset currents on $F_N$ (Lemma \ref{scyl,occ}). 

Recall that $R_N$ is the quotient graph $F_N\backslash X$, which is an $N$-rose. An $R_N$-\ti{graph} is a graph $\gD$ with a graph morphism $\tau \: \gD \rightarrow R_N$. We call $\tau$ an $R_N$-graph structure.
Let $(\gD_1,\tau_1)$ and $(\gD_2, \tau_2)$ be $R_N$-graphs. A graph morphism $f\: \gD_1\rightarrow \gD_2$ is called an $R_N$-\ti{graph morphism} if $\tau_1=\tau_2 \circ f$. For an $R_N$-graph $(\gD ,\tau )$ and $v\in V(\gD )$ we call a pair $((\gD ,\tau ),v)$ (or simply $(\gD ,v)$) a based $R_N $-graph with a base point $v$.

An $R_N$-graph $(\gD, \tau)$ is said to be \ti{folded} if $\tau$ is locally injective (immersion). 
We say that a finite $R_N$-graph $(\gD ,\tau)$ is an $R_N$-\ti{core graph} if $(\gD ,\tau )$ is folded and has no degree-one and degree-zero vertices. We do not assume that $R_N$-core graphs are connected. 

For a graph $\gD$ giving an $R_N$-structure $\tau \: \gD \rightarrow R_N$ to $\gD$ can be regarded as giving a \ti{label structure} $E_{\tn{top}}(\gD )\to A$ and giving an orientation to $\gD$. We say a topological or oriented edge $e$ of $\gD$ has a label $a\in A$ when the graph morphism $\tau$ maps the edge $e$ to the loop of $R_N$ corresponding to $a\in A$.

\begin{definition}{\em
Let $\gD$ be a graph. Let $T\in \Sub (X)$. The \ti{interior} of $T$ is $T\setminus \{ \tn{degree-one vertices of}\ T\}$. A graph morphism $f\: T\to \gD$ is said to be \ti{locally homeomorphic in the interior} if the restriction of $f$ to the interior of $T$ is locally homeomorphic, in other words, the degree of $v$ in $T$ equals to the degree of $f(v)$ in $\gD$ for every $v\in V(T)$ with degree more than one.

Let $T\in \tn{Sub}(X)$, and $Y$ be a (not necessarily finite) subtree of $X$. We say that $Y$ is an \ti{extension} of $T$ and denote by
$T \underset{\mathrm{ext}}{\subset }Y$, if $T\subset Y$ and if the inclusion map is locally homeomorphic in the interior.
}\end{definition}

\begin{proposition}\label{scyl}
Let $S\in \C_N$ and $T\in \Sub(X)$. Then $S\in \SCyl (T)$ if and only if $T\underset{\mathrm{ext}}{\subset }\mathrm{Conv}(S)$, where $\mathrm{Conv}(S)$ is the convex hull of $S$ in $X$.
\end{proposition}
\begin{proof}
Let $e_1,\dots ,e_m$ be all the terminal edges of $T$ and $v_i$ be the terminal of $e_i$.
Then we have 
\[ T =\bigcup _{i,j=1,\dots ,m}[v_i,v_j],\]
where $[v_i , v_j ]$ is the geodesic from $v_i$ to $v_j $ in $X$. Similarly,
\[ \mathrm{Conv}(S)=\bigcup _{\xi ,\zeta \in S} (\xi , \zeta ),\]
where $(\xi , \zeta )$ is the bi-infinite geodesic from $\xi$ to $\zeta $ in $X$.

Suppose $S\in \SCyl (T)$. By the definition $S\subset \bigcup_i \tn{Cyl}(e_i)$ and $S\cap \tn{Cyl}(e_i)\not= \emptyset \ (i=1,\dots , m)$. If $\xi ,\zeta \in S\cap \tn{Cyl}(e_i)$ for the same $i$, then $(\xi, \zeta )$ does not contain any edges of $T$. If $\xi \in S\cap \tn{Cyl}(e_i), \zeta \in S\cap \tn{Cyl}(e_j) (i\not=j)$, then the geodesic $(\xi, \zeta)$ is an extension of $[v_i, v_j]$. Therefore we have $T\underset{\mathrm{ext}}{\subset }\mathrm{Conv}(S)$.

Next, we assume $T\underset{\mathrm{ext}}{\subset }\mathrm{Conv}(S)$. For every $v_i, v_j(i\not=j )$ there exist $\xi ,\zeta \in S$ such that $(\xi ,\zeta )$ is an extension of $[v_i, v_j]$ and then we have $\xi \in \tn{Cyl}(e_i)$ and $\zeta \in \tn{Cyl}(e_j)$. Hence $S\cap \tn{Cyl}(e_i)\not=\emptyset\ (i=1,\dots ,m)$. To prove $S\subset \bigcup_i \tn{Cyl}(e_i)$, we take $\xi \in S$ and $\zeta \in S\cap \tn{Cyl}(e_i)$ for some $i$. If $(\xi ,\zeta)$ contains some edges of $T$, then $T\underset{\mathrm{ext}}{\subset }\mathrm{Conv}(S)$ implies that there exists $v_j$ such that $(\xi ,\zeta)$ is an extension of $[v_j, v_i]$. Thus $\xi \in \tn{Cyl}(e_j)$ and $S\subset \bigcup_i \tn{Cyl}(e_i)$.
If $(\xi ,\zeta)$ does not contain any edges of $T$, we choose $\zeta' \in S\cap \tn{Cyl}(e_j) \ (j\not=i)$. Then $(\zeta ,\zeta')$ is an extension of $[v_i,v_j]$. Since $(\xi ,\zeta)\cup (\xi, \zeta')\supset (\zeta, \zeta')$, the geodesic $(\xi,\zeta')$ contains some edges of $T$. Now, we can apply the above argument to $(\xi ,\zeta')$.
\end{proof}

\begin{definition}{\em
Let $T\in \Sub (X)$. We can regard $T$ as an $R_N$-graph with the $R_N$-graph structure inherited from the canonical projection $q\: X\to R_N$. Let $\gD$ be an $R_N$-core graph. An \ti{occurrence} of $T$ in $\gD$ is an $R_N$-graph morphism $f\: T\rightarrow \gD$ which is locally homeomorphic in the interior.
Let $\mathrm{Occ}(T, \gD )$ be the set of all occurrences of $T$ in $\gD$.
}\end{definition}

Consider a based $R_N$-graph $(T, x)$ such that $T\in \Sub (X)$. For a based and folded $R_N$-graph $(\gD ,v)$ there exists at most one based occurrence $f\: (T,x)\rightarrow (\gD,v)$, where $f(x)=v$. In order to compute $\# \mathrm{Occ}(T,\gD )$, it is sufficient to see whether there exists a based $R_N$-graph morphism $f\: (T,x)\rightarrow (\gD, v)$ for each $v\in V(\gD )$. Hence 
\[ \# \mathrm{Occ}(T,\gD) =\# \{ v\in V(\gD )\ |\ \exists f\: (T,x)\rightarrow (\gD , v)\  \tn{a based occurrence}\}.\]

\begin{notation}{\em
Let $H\in \Sub(F_N)$. Set $X_H:=\mathrm{Conv}(\Lambda (H))$, which is the unique minimal $H$-invariant subtree of $X$. We define $\gD_H$ to be the quotient graph $H\backslash X_H$ and denote by $q_H\: X_H\rightarrow \gD_H$ the canonical projection. Then $\gD_H$ becomes an $R_N$-core graph by the induced graph morphism $\tau_H\: \gD_H\rightarrow R_N$ from $q\:X\rightarrow R_N$. 

We denote by $\Sub (X,\id )$ the set of all nontrivial finite based subtrees of $X$ with the base point $\id$. For $(T,\id )\in \Sub (X,\id )$ we denote it briefly by $T$.
}\end{notation}

The following lemma is a direct corollary from Section 4.2 and Section 4.3 in \cite{KN13}, and plays an essential role in studying rational subset currents on $F_N$.

\begin{lemma}\label{scyl,occ}
Let $H\in \Sub(F)$ and $T\in \Sub(X,\id )$. Then we have 
\begin{align*}
\eta_H(\SCyl (T))
	&=\# \mathrm{Occ}(T,\gD_H)\\
	&(=\# \{ v\in V(\gD_H )\ |\ \exists f\: (T,\id )\rightarrow (\gD_H, v)\  \tn{a based occurrence}\} ).
\end{align*}
\end{lemma}

\subsection{The reduced rank functional $\rk$}
Recall that the \ti{rank} $\tn{rank}(F)$ of a finitely generated free group $F$ is the cardinality of a free basis of $F$, and the \ti{reduced rank} $\rk (F)$ is defined as 
\[ \rk (F):=\tn{max}\{ \tn{rank}(F)-1,0\}.\]
If $\tn{rank}(F)\geq 1$, then $\rk (F)=\tn{rk}(F)-1$, and for a finite connected graph $\gD$ whose fundamental group is isomorphic to $F$ we have $\rk(F)=-\chi (\gD)$, where $\chi (\gD):=\# V(\gD)-\# E_{\tn{top}}(\gD)$ is the Euler characteristic of $\gD$.

\begin{theorem}[cf. \cite{KN13}, Theorem 8.1]\label{thm:reduced rank}
There exists a unique continuous $\RRR$-linear functional
\[ \rk \colon \SCurr (F_N)\rightarrow \RRR\]
such that for every $H\in \Sub (F_N)$ we have 
\[ \rk (\eta_H )=\rk (H).\]
Moreover, $\rk$ is $\tn{Out}(F_N)$-invariant.
\end{theorem}
The functional $\rk$ is called the \ti{reduced rank functional}.
In order to prove Theorem \ref{thm:Neumann product} we will use a method used in the proof of \cite[Theorem 8.1]{KN13}.
We give a proof of Theorem \ref{thm:reduced rank} following the argument in \cite{KN13} almost step by step in the remaining part of this subsection.

Note that for $H\in \Sub (F_N)$ we have $\tn{rank}(H)\geq 1$, and so
\[ \rk (H)=\# E_{\tn{top}}(\gD_H )- \# V(\gD_H ).\]
We extend each term of the right hand side of this equation to a continuous $\RRR$-linear functional on $\SCurr (F_N)$, namely construct continuous $\RRR$-linear functionals
\[ E, V\: \SCurr (F_N)  \rightarrow  \RRR \]
such that $E(\eta_H)= \# E_{\tn{top}}(\gD_H)$ and $V(\eta_H )=\# V(\gD_ H)$ for $H\in \Sub (F_N)$. 

Let $e_a$ be the topological edge in $X$ with the boundary $\{\id ,a\} \ (a \in A)$. Let $H\in \Sub (F_N)$. By considering $e_a$ as a subtree of $X$ we have $(e_a,\id )\in \Sub (X,\id )$. Then $\eta_H (\SCyl (e_a))=\# \mathrm{Occ} (e_a, \gD_H)$ coincides with the number of topological edges of $\gD_H$ with the label $a$. From this we define the map

\[ E\:  \SCurr (F_N)  \rightarrow  \RRR ;\ 
	 \mu  \mapsto \sum_{a\in A} \mu (\SCyl (e_a)).\]
Then we have $E(\eta_H)=\# E_{\tn{top}}(\gD_H)$ for $H\in \Sub (F_N)$, and $E$ is continuous and $\RRR $-linear from Proposition \ref{prop:Scyl is conti}.

In the case of $V$, we need the following definition.

\begin{definition}[Round graphs. cf. \cite{Kap13}, Definition 3.6]\label{rem:r-neighborhood}{\em 
For an integer $r\geq 1$, we say that $T\in \Sub (X)$ is a \ti{round graph} of \ti{grade} $r$ in $X$ if there exists a (necessarily unique) vertex $v$ of $T$ such that for every degree-one vertex $u$ of $T$ we have $d_X(v, u)=r$. We call a pair $(T,v)$ a based round graph of grade $r$ if $v$ satisfies the above condition, where the degree of $v$ have to be more than one. Let $\R_r $ denote the set of all based round graphs of grade $r$ with the base point $\id$. Thus $\R_r$ is a subset of $\Sub (X,\id)$.
}\end{definition}

\begin{remark}{\em
Fix a positive integer $r$. For any $R_N$-graph $\gD$ and $v\in V(\gD )$ there exists a unique $T_r(v)\in \R_r$ such that there exists a based occurrence $(T_r(v), \id )\rightarrow (\gD ,v)$. We can think of $T_r(v)$ as an $r$-\ti{neighborhood} of $v$ in $\gD$.

Let $v\in V(X)$ and $\rho \in \RRR$. Set $B(v ,\rho ):=\{ x\in X \ |\ d_X(v,x)\leq \rho \}$ the closed ball with radius $\rho$ and center $v$ in $X$. For $T\in \R_r$ and every $S\in \SCyl (T)$ we have $\mathrm{Conv}(S)\cap B(\id , r)=T$ from Proposition \ref{scyl}.
Therefore, if $T_1\not=T_2 $ for $T_1,T_2\in \R_r$, then $\SCyl (T_1)\cap \SCyl (T_2) =\emptyset $.
}\end{remark}

We define the map
\[ V\:  \SCurr (F_N)  \rightarrow  \RRR ;\ 
	 \mu  \mapsto \sum_{T\in \R_1} \mu (\SCyl (T))=\mu \left( \bigsqcup_{T\in \R_1}\SCyl (T)\right).\]
Then
\begin{align*}
 V(\eta_H )=&\sum_{T\in \R_1} \eta_H (\SCyl (T)) \\
	=&\sum_{T\in \R_1} \# \{ v\in V(\gD_H)\ |\ T_1(v)=T\} =\# V(\gD _H) 
\end{align*}
for $H\in \Sub (F_N)$, and $V$ is continuous and $\RRR$-linear. Note that for any positive integer $r$, we have
\[ \bigsqcup_{T\in \R_r}\SCyl (T)=\{ S\in \C_N \ |\ \mathrm{Conv}(S)\ni \id \}.\]

\begin{proof}[Proof of Theorem \ref{thm:reduced rank}]
Set $\rk =E-V$. Then for $H\in \Sub (F_N)$,
\[ \rk (\eta_H )=E(\eta_H )-V(\eta_H )=\# E(\gD_H )- \# V(\gD_H )=\rk (H).\]
Since $\rk (r\eta_H)=r\rk (H)\geq 0$ for any rational subset current $r \eta_H (r\geq 0, H\in \Sub (F_N))$, we have $\rk (\mu )\geq 0$ for any $\mu \in \SCurr (F_N)$. The uniqueness and $\tn{Out}(F_N)$-invariance of $\rk$ is obvious from the denseness of the rational subset currents in $\SCurr (F_N)$.
\end{proof}

\section{The intersection functional $\N$}\label{sec:product}
Define a product $\N (H,K)$ of two finitely generated subgroups $H$ and $K$ of $F_N$ as
\[ \N (H,K):=\sum_{HgK\in H\backslash F_N /K } \rk (H\cap gKg^{-1}), \]
where $H\backslash F_N /K$ is the set of all double cosets $HgK\ (g\in F_N)$. By this definition the Strengthened Hanna Neumann Conjecture (SHNC), which has been proven independently by Friedman and Mineyev, can be stated as follows:

\begin{theorem}[SHNC, see \cite{Fri11, Min12a}]\label{thm:Mineyev}
For any finitely generated subgroups $H,K\leq F_N$, the following inequality holds:
\[ \N (H,K) \leq \rk (H) \rk (K).\]
\end{theorem}

In this section we give a proof of the following theorem, which is our main result.

\begin{theorem}\label{thm:Neumann product}
There exists a unique continuous symmetric $\RRR$-bilinear functional
\[ \N \: \SCurr (F_N)\times \SCurr (F_N)\rightarrow \RRR\]
such that for $H,K\in \Sub (F_N)$ we have
\[ \N (\eta_H ,\eta_K )= \N (H,K).\]
Moreover, $\N$ is $\tn{Out}(F_N)$-invariant, that is, for any $\varphi \in \tn{Out}(F_N)$ and $\mu ,\nu \in \SCurr (F_N)$ we have $\N (\varphi \mu ,\varphi \nu )=\N (\mu ,\nu )$.
\end{theorem}
We call $\N$ the \ti{intersection functional}.
By the definition of the product $\N$ and the $\RRR$-linearity of the reduced rank functional $\rk$, we have
\begin{align*}
\N (H,K)&=\sum_{HgK\in H\backslash F_N /K } \rk (H\cap gKg^{-1}) \\
	&=\rk \left( \sum_{HgK\in H\backslash F_N/K} \eta_{H\cap gKg^{-1}}\right).
\end{align*}
Kapovich and Nagnibeda asked whether there exists a continuous $\RRR$-bilinear map
\[ \pitchfork \: \SCurr (F_N)\times \SCurr (F_N)\to \SCurr (F_N)\]
such that 
\[ \pitchfork (\eta _H, \eta _K )=\sum_{HgK\in H\backslash F_N/K} \eta_{H\cap gKg^{-1}}\]
for any finitely generated subgroups $H,K\leq F_N$ (see \cite[Subsection 10.4]{KN13}). If such a map $\pitchfork$ exists, then Theorem \ref{thm:Neumann product} follows immediately, and moreover, $\pitchfork$ can be considered as a measure theoretical generalization of the construction of the fiber product graph $\gD_H \times_{R_N}\gD_K$ (see Definition \ref{def:fiber product graph} and the following argument).
However, by the following proposition we answer that there does not exists such a map $\pitchfork$. Nevertheless we construct a non-continuous $\RRR$-bilinear map $\widehat{I}\: \SCurr (F_N)\times \SCurr (F_N)\to \SCurr (F_N)$ with reasonable properties (see Section \ref{sec:intersection}).

\begin{proposition}\label{prop:not conti}
If an $\RRR$-bilinear map
\[ \pitchfork \: \SCurr (F_N)\times \SCurr (F_N)\to \SCurr (F_N)\]
satisfies the condition that
\[ \pitchfork (\eta _H, \eta _K )=\sum_{HgK\in H\backslash F_N/K} \eta_{H\cap gKg^{-1}}\]
for any finitely generated subgroups $H,K\leq F_N$, then $\pitchfork$ is not continuous.
\end{proposition}
\begin{proof}
Recall that $A=\{ a_1,\dots ,a_N\}$ is a free basis of $F_N$.
Set subgroups $H_n:=\langle a_1^{n} a_2 \rangle \ (n=1,2,\dots )$ and $H:= \langle a_1 \rangle $. Then from Proposition \ref{prop:Scyl is conti} and Lemma \ref{scyl,occ} we can see 
\[ \frac{1}{n}\eta_{H_n}\rightarrow \eta_{H}\ (n\rightarrow \infty ).\]
We also have
\[ \pitchfork (\eta_{H_n}, \eta_{H})=0 \ (n=1,2,\dots )\ \tn{and}\ \pitchfork(\eta_H ,\eta_H )=\eta_H , \]
which implies that $\pitchfork(\frac{1}{n}\eta_{H_n},\eta_H)=\frac{1}{n}\pitchfork (\eta_{H_n}, \eta_{H})$ does not converges to $\pitchfork(\eta_H ,\eta_H )$. Therefore, $\pitchfork$ is not continuous.
\end{proof}

To prove Theorem \ref{thm:Neumann product} we use the following graph theoretical description of $\N$. 
First, we recall the definition of the \ti{fiber product graph} in \cite{Sta83}. Since by a graph we mean a $0$ or $1$-dimensional CW complex, we rearrange the definition.

\begin{definition}[cf. \cite{Sta83}]\label{def:fiber product graph}{\em
Let $(\gD_1,\tau_1),(\gD_2,\tau_2)$ be $R_N$-graphs. The \ti{fiber product graph} $\gD_1\times_{R_N}\gD_2$ corresponding to $(\gD_1,\tau_1)$ and $(\gD_2 ,\tau_2)$ is the fiber product of $(\gD_1,\tau_1)$ and $(\gD_2 ,\tau_2)$ in the category of topological spaces with a graph structure induced by $(\gD_1,\tau_1)$ and $(\gD_2 ,\tau_2)$. Explicitly, 
\begin{align*}
	\gD_1\times_{R_N}\gD_2 &=\{ (x_1,x_2)\in \gD_1 \times \gD_2 \ |\ \tau_1 (x_1)=\tau_2 (x_2)\};\\
	V(\gD_{1}\times _{R_N}\gD_{2})&=\{ (v_1,v_2)\in V(\gD_1 )\times V(\gD_2 )\}; \\
	E_{\tn{top}}(\gD_{1}\times _{R_N}\gD_{2})&=\{ (e_1,e_2)\in E_{\tn{top}}(\gD_1 )\times E_{\tn{top}}(\gD_2 )\ |\ \tau_1 (e_1)=\tau_2 (e_2)\}. 
\end{align*}
Here, $V(\gD_{1}\times _{R_N}\gD_{2})$ is given as above because $R_N$ has only one vertex. For $(e_1,e_2)\in E_{\tn{top}}(\gD_1\times_{R_N}\gD_2)$, fix orientations of $e_1, e_2$ such that $\tau_1(e_1)=\tau_2(e_2)$ in $E(R_N)$. Then the boundary of $(e_1,e_2)$ is $\{ (o(e_1),o(e_2)), (t(e_1),t(e_2))\}$, which does not depend on the choice of the orientations of $e_1$ and $e_2$. 

There are natural graph morphisms $\phi_i\: \gD_1\times_{R_N}\gD_2 \to \gD_i, (x_1, x_2)\to x_i \ (i=1,2)$, and the graph morphism $\phi_i$ induces
\begin{align*}
\phi_i &\:V(\gD_1\times_{R_N}\gD_2 )\to V(\gD_i );\ (v_1,v_2)\mapsto v_i,\\
\phi_i &\:E_{\tn{top}}(\gD_1\times_{R_N}\gD_2 )\to E_{\tn{top}}(\gD_i );\ (e_1,e_2)\mapsto e_i.
\end{align*}

The fiber product graph $\gD_1\times_{R_N}\gD_2$ becomes an $R_N$-graph by the map $\tau_1\circ \phi_1(=\tau_2\circ \phi_2)$. If $(\gD_1,\tau_1)$ and $(\gD_2,\tau_2)$ are folded $R_N$-graphs, then so is $(\gD_1\times_{R_N}\gD_2, \tau_1\circ \phi_1 )$. However, $\gD_1\times_{R_N}\gD_2$ may not be connected and may have degree-zero and degree-one vertices even if $\gD_1$ and $\gD_2$ are connected $R_N$-core graphs.
}\end{definition}

Let $H,K\in \Sub  (F_N)$. In the context of the SHNC it has been proven that every non-zero term in the sum of $\N (H,K)$ corresponds to a non-contractible component of $\gD_H \times_{R_N}\gD_K$, and the following equality holds (see \cite{Neu90}):
\[ \N (H,K)=-\sum _{i=1}^k \chi (\gG_i ) ,\]
where $\gG_1,\dots ,\gG_k$ are all the non-contractible components of $\gD_H \times_{R_N} \gD_K$.
Using this description of $\N$ our strategy of proving Theorem \ref{thm:Neumann product} is the same as that of Theorem \ref{thm:reduced rank}.
We denote by $c(\gD_{H}\times _{R_N}\gD_{K})$ the number of contractible components of $\gD_{H}\times _{R_N}\gD_{K}$. Since the Euler characteristic of a contractible component is $1$, we have 
\[ \N (H,K)=\# E_{\tn{top}}(\gD_{H}\times _{R_N}\gD_{K})-\# V(\gD_{H}\times _{R_N}\gD_{K}) +c(\gD_{H}\times _{R_N}\gD_{K}).\]
We extend each term of the right hand side of this equation to a continuous symmetric $\RRR $-bilinear functional on $\SCurr (F_N)$.

Note that for $T_1,T_2\in \Sub (X)$ we have a continuous $\RRR $-bilinear functional
\[ \SCurr (F_N)\times \SCurr (F_N)\rightarrow \RRR ;\ (\mu ,\nu )\mapsto \mu(\SCyl (T_1))\nu (\SCyl (T_2)),\]
which plays a fundamental role in constructing the intersection functional $\N$.

Since $\# V(\gD_{H}\times _{R_N}\gD_{K})=\# V(\gD_{H})\# V(\gD_{K})$, we define the continuous $\RRR$-bilinear map
\[ \widehat{V} \: \SCurr (F_N)\times \SCurr (F_N)\rightarrow \RRR \]
by
\begin{align*}
 \widehat{V} (\mu, \nu )
	:=&V(\mu )V(\nu ) \\
	=&\sum_{T\in \R_1} \mu (\SCyl (T))\sum_{T\in \R_1} \nu (\SCyl (T))\\
	=&\sum_{T,T'\in \R_1} \mu (\SCyl (T))\nu (\SCyl (T)).
\end{align*}
Then we have $\widehat{V}(\eta_H, \eta_K)=\# V(\gD_{H}\times _{R_N}\gD_{K})$.

In the case of $\# E_{\tn{top}}(\gD_H \times_{R_N}\gD_K)$, for each $a\in A$ the number of topological edges of $\gD_{H}\times _{R_N}\gD_{K}$ with the label $a$ equals to the product of the number of topological edges of $\gD_H$ with the label $a$ and that of $\gD_K$ with the label $a$. Therefore, we have
\[ \# E_{\tn{top}}(\gD_{H}\times _{R_N}\gD_{K})=\sum_{a\in A}\eta_H (\SCyl (e_a))\eta_H (\SCyl (e_a)).\]
Now, we define the continuous $\RRR$-bilinear map
\[ \widehat{E} \: \SCurr (F_N)\times \SCurr (F_N)\rightarrow \RRR \]
by
\[ \widehat{E}(\mu ,\nu) :=\sum_{a\in A}\mu (\SCyl (e_a))\nu (\SCyl (e_a)).\]
Then we have $\widehat{E}(\eta_H, \eta_K)=\# E_{\tn{top}}(\gD_{H}\times _{R_N}\gD_{K})$.

In the remaining part of this section we construct a continuous $\RRR$-bilinear functional 
\[\widehat{c}\: \SCurr (F_N)\times \SCurr (F_N)\to \RRR\]
such that $\widehat{c}(\eta_H,\eta_K)=c(\gD_{H}\times _{R_N}\gD_{K})$.

Set $\SSub (X,\id ):=\Sub (X,\id )\cup \{ \{ \id\} \}$, where $\{\id \}$ is regarded as the subtree of $X$ consisting of one vertex $\id$.

To construct $\widehat{c}$ we use the following lemmas.
First one is obvious from the definition of the fiber product graph.

\begin{lemma}\label{lem:occ and int}
Let $(\gD_1,\tau_1),(\gD_2,\tau_2)$ be $R_N$-core graphs and $v_i \in V(\gD_i)\ (i=1,2)$. 
Let $\gG$ be the connected component of $\gD_1\times_{R_N}\gD_2$ containing the vertex $(v_1,v_2)$. Let $f_i\: (T_i,\id )\to (\gD_i ,v_i)\ (T_i \in \Sub (X, \id))$ be a based occurrence $(i=1,2)$, and set $T=T_1\cap T_2$. Then we have an $R_N$-graph morphism
\[ f\: T\to \gD_1 \times_{R_N} \gD_2 ;\ x\mapsto (f_1(x),f_2(x)),\]
and $f$ is locally homeomorphic in the interior.
\end{lemma}

\begin{lemma}\label{lem:contra}
Let $T\in \SSub (X,\id )$ such that $T\subset B(\id ,r)$ for an integer $r\geq 0$. Let $T_1,T_2 \in \R_{r+1}$ with $T_1\cap T_2=T$. Let $(\gD_1,\tau_1),(\gD_2,\tau_2)$ be $R_N$-core graphs and $v_i$ a vertex of $\gD_i \ (i=1,2)$.
If there are based occurrences $f_i\: (T_i ,\id )\rightarrow (\gD_i,v_i)\ (i=1,2)$, then there exists an $R_N$-graph isomorphism from $T$ to the connected component $\gG$ of $\gD_1\times_{R_N}\gD_2$ containing the vertex $(v_1,v_2)$. In particular, $\gG$ is contractible.
\end{lemma}
\begin{proof}
We denote by $\phi_i$ the natural $R_N$-graph morphism from $\gD_1\times_{R_N}\gD_2$ to $\gD_i$ ($i=1,2$), and consider the $R_N$-graph morphism 
$f \: T\rightarrow \gD_1\times_{R_N}\gD_2 ;\ x \mapsto (f_1(x),f_2(x))$ given by Lemma \ref{lem:occ and int}.
Since $f(\id )=(f_1(\id ),f_2(\id ))=(v_1,v_2)$, we have $f(T)\subset \gG $. We prove that the map $f\: T\rightarrow \gG $ is an $R_N$-graph isomorphism. 

First, we prove the surjectivity of $f$. Take any $x\in \gG$ and a locally isometric path $p\: [0,l]\rightarrow \gG$ such that $p(0)=(v_1,v_2)$ and $p(l)=x$. Assume that $l\leq r+1$. Since $f_i$ is locally homeomorphic in the interior, for the path $\phi_i \circ p$ in $\gD_i$ we can take the lift $\tilde{p}_i\: [0,l]\rightarrow T_i$ such that $f_i\circ \tilde{p_i}=\phi_i \circ p$ and $\tilde{p}_i(0)=\id$ ($i=1,2$). Then $\tilde{p}_i$ can be regarded as the lift of 
$\tau_i\circ \phi _i\circ p\: ([0,l],0)\to (R_N ,x_0)$ with respect to the universal covering $q\: (X,\id )\rightarrow (R_N,x_0)$. Since $\tau_1\circ \phi _1\circ p=\tau_2\circ \phi _2\circ p$, we have $\tilde{p}_1=\tilde{p}_2$ from the uniqueness of the lift.
Therefore $\tilde{p}_1(l)=\tilde{p}_2(l)\in T_1\cap T_2=T$, and $f(\tilde{p}_1(l))=(\phi_1\circ p(l),\phi_2 \circ p(l))=x$. Hence $f$ is surjective, and moreover, it follows that $l\leq r$.
If the length $l$ of $p$ is greater than $r+1$, we can do the same argument for $p|_{[0,r+1]}$ and this leads to a contradiction.
Consequently, there is no locally isometric path starting from $(v_1,v_2)$ in $\gG$ with length greater than $r$. In particular, $\gG$ is a tree.

Since $f$ is an $R_N$-graph morphism from the tree $T$ to the tree $\gG$, the injectivity of $f$ follows. Therefore, $f$ is an $R_N$-graph isomorphism from $T$ to $\gG$.
\end{proof}

\begin{lemma}\label{T_1,T_2}
Let $T\in \SSub (X,\id)$ with $T\subset B(\id ,r) $. Let $(\gD_1,\tau_1)$ and $(\gD_2,\tau_2)$ be $R_N$-core graphs and $\gG$ the connected component of $\gD_1\times_{R_N}\gD_2$ containing a vertex $(v_1,v_2)\in V(\gD_1 )\times V(\gD_2)$. Then there exists a based $R_N$-graph isomorphism from $(T,\id )$ to $(\gG , (v_1,v_2))$ if and only if $T_{r+1}(v_1)\cap T_{r+1}(v_2)=T$, where $T_{r+1}(v_i)$ is $(r+1)$-neighborhood of $v_i\ (i=1,2)$.
\end{lemma}
\begin{proof} The ``if'' part follows from Lemma \ref{lem:contra}. To prove the ``only if'' part, we suppose there exists an $R_N$-graph isomorphism $\varphi \: (T,\id )\to (\gG , (v_1,v_2))$.
Let $f_i\: (T_{r+1}(v_i),\id )\to (\gD_i ,v_i)$ be the occurrence ($i=1,2$). From Lemma \ref{lem:occ and int} we have the $R_N$-graph morphism
\[f\: T_{r+1}(v_1)\cap T_{r+1}(v_2)\to \gG \subset \gD_1 \times _{R_N} \gD_2 ;\ x\to (f_1(x),f_2(x))\]
which is locally homeomorphic in the interior. Then the $R_N$-graph morphism $\varphi ^{-1}\circ f\: T_{r+1}(v_1)\cap T_{r+1}(v_2)\to T$ is locally homeomorphic in the interior and $\varphi^{-1}\circ f (\id )=\id$, which implies that $T_{r+1}(v_1)\cap T_{r+1}(v_2)\subset T$. Since $T\subset B(\id ,r)$, for the natural $R_N$-graph morphism $\phi_i \: \gD_1\times_{R_N}\gD_2 \to \gD_i$ we have the lift $\tilde{\phi }_i\: T\to T_{r+1}(v_i)$ such that $f_i\circ \tilde{\phi}_i=\phi_i \circ \varphi$ and $\tilde{\phi}_i(\id )=\id$ ($i=1,2$). Therefore, 
$T\subset T_{r+1}(v_1)\cap T_{r+1}(v_2)$, which concludes that $T_{r+1}(v_1)\cap T_{r+1}(v_2)=T$.
\end{proof}

\begin{notation}{\em
Let $T\in \SSub (X,\id)$ and let $(\gD_1,\tau_1),(\gD_2,\tau_2)$ be $R_N$-core graphs. We denote by $c(\gD_1\times_{R_N}\gD_2, T)$ the number of contractible components of $\gD_1\times_{R_N}\gD_2$ that are $R_N$-graph isomorphic to $T$.

We define $\SSub (X,\id )/F_N$ to be the set of all the equivalence classes of the following equivalence relation: $T_1\in \SSub (X,\id )$ is equivalent to $T_2\in \SSub (X,\id )$ if there exists $g\in F_N$ such that $gT_1=T_2$. We denote by $[T]$ the equivalence class containing $T\in \SSub (X,\id)$.

For any contractible component $\gG$ of $\gD_1\times_{R_N}\gD_2$ we have a lift $\gG \to X$ of $\tau_1 \circ \phi_1\: \gG \to R_N$ with respect to the universal covering $q\: X\to R_N$, which implies that there exists a unique equivalence class $[T]\in \SSub (X,\id )/F_N$ such that $\gG$ is $R_N$-graph isomorphic to $T$. Therefore we have
\[ c(\gD_1\times_{R_N}\gD_2 )=\sum_{[T]\in \SSub(X,\id )/F_N} c(\gD_1\times_{R_N}\gD_2 ,T).\]
}\end{notation}

Let $H,K\in \Sub (F_N)$ and $T\in \SSub(X,\id )$ with $T\subset B(\id ,r)$ for an integer $r\geq 0$. From Lemma \ref{T_1,T_2} and Lemma \ref{scyl,occ}, we obtain
\begin{align*}
&c(\gD_{H}\times_{R_N}\gD_{K},T)\\
	=&\sum_{\substack{T_1,T_2\in \R_{r+1}\\[1pt] T_1\cap T_2=T}}\# \{v\in V(\gD_{H}): T_{r+1}(v)=T_1\} \cdot \# \{v\in V(\gD_{K}): T_{r+1}(v)=T_2\}\\
	=&\sum_{\substack{T_1,T_2\in \R_{r+1}\\[1pt] T_1\cap T_2=T}}\eta_{H}(\SCyl (T_1))\eta_{K}(\SCyl (T_2)).
\end{align*}

\begin{notation}{\em
Let $T\in \SSub (X,\id)$. Set
\[ \R (T):=\{ (S_1,S_2)\in \C_N\times \C_N \ |\ \mathrm{Conv}(S_1)\cap \mathrm{Conv}(S_2)=T\}.\]
Let $r$ be a positive integer which satisfies $T\subset B(\id ,r)$. Then we can see that the following equality holds:
\[ (*)\qquad \R (T)=\bigsqcup_{\substack{T_1,T_2\in \R_{r+1}\\[1pt] T_1\cap T_2=T}}\SCyl (T_1)\times \SCyl (T_2),
\]
and so 
\[ c(\gD_{H}\times_{R_N}\gD_{K}, T)= \eta_{H}\times \eta_{K}(\R (T)),\]
where $\eta_H \times \eta_K$ is the product measure of $\eta_H$ and $\eta_K$.
}\end{notation}

From the definition of $\R (T)$ and $(*)$, we immediately have the following proposition.

\begin{proposition}Let $T,T'\in \SSub (X,\id )$.
\begin{enumerate}
\item If $T\not=T'$, then we have 
\[ \R (T)\cap \R (T')=\emptyset.\]
\item If there exists $g\in F_N$ such that $gT=T'$, then for any $\mu,\nu \in \SCurr (F_N)$
\[ \mu \times \nu (\R (T))=\mu \times \nu (\R (T')).\]
\item The map
\[ \SCurr (F_N)\times \SCurr (F_N)\to \RRR ;\ (\mu ,\nu )\mapsto \mu \times \nu (\R (T))\]
is a continuous $\RRR$-bilinear map.
\end{enumerate}
\end{proposition}

Moreover, $\R (T)\ (T\in \SSub (X,\id ))$ has the following properties:
\begin{align*}
&\bigsqcup_{T\in \SSub (X,\id )} \R (T) \\
	=&\{ (S_1,S_2)\in \C_N \times \C_N \ |\ \exists T\in \SSub (X,\id ), \mathrm{Conv}(S_1)\cap \mathrm{Conv}(S_2)=T\} \\
	\subset &\{ (S_1,S_2)\in \C_N \times \C_N \ |\ \mathrm{Conv}(S_1),\mathrm{Conv}(S_2)\ni\id \} \\
	=&\bigsqcup_{T_1,T_2\in \R _1}\SCyl (T_1)\times \SCyl (T_2).
\end{align*}
Therefore, for $\mu ,\nu \in \SCurr (F_N)$ 
\begin{align*}
\sum_{T\in \SSub (X,\id )} \mu \times \nu (\R (T)) 
	&=\mu \times \nu \left( \bigsqcup_{T\in \SSub (X,\id )} \R (T)\right)\\
	&\leq \mu \times \nu \left( \bigsqcup_{T_1,T_2\in \R _1}\SCyl (T_1)\times \SCyl (T_2) \right) \\
	&=\mu \left(\bigsqcup_{T_1\in \R_1} \SCyl (T_1)\right) \nu \left(\bigsqcup_{T_2\in \R_1} \SCyl (T_2)\right) \\
	&=V(\mu )V(\nu ).
\end{align*}
Hence the infinite sum 
\[\sum_{T\in \SSub (X,\id )} \mu \times \nu (\R (T))  \]
always converges.
Since $\# [T]=\# V(T)$ for $T\in \SSub (X,\id )$, we have
\[  \sum_{T\in \SSub (X,\id )} \mu \times \nu (\R (T))
	=\sum_{[T]\in \SSub (X,\id )/F_N} \# V(T) \mu \times \nu (\R (T)). \]

Let $H,K\in \Sub (F_N)$. Then
\begin{align*}
c(\gD_{H}\times_{R_N}\gD_{K})
	&=\sum_{[T]\in \SSub (X,\id )/F_N} c(\gD_{H}\times_{R_N}\gD_{K}, T)\\
	&=\sum_{[T]\in \SSub (X,\id )/F_N} \eta_{H}\times \eta_{K}(\R (T)).
\end{align*}
From the above, we define the $\RRR$-bilinear map 
\[ \widehat{c}\:  \SCurr (F_N)\times \SCurr (F_N)\rightarrow \RRR \]
as
\[ \widehat{c}(\mu ,\nu ):= \sum_{[T]\in \SSub (X,\id )/F_N} \mu \times \nu (\R (T))\ .\]
Then we immediately have $\widehat{c} (\eta_{H},\eta_{K})=c(\gD_{H}\times_{R_N}\gD_{K})$ for $H,K\in \Sub (F_N)$.

Note that $\widehat{c}(\mu,\nu)$ can be represented by 
\[ \widehat{c}(\mu ,\nu )=\sum_{m=1}^{\infty } \sum_{\substack{[T]\in \SSub (X,\id )/F_N\\[1pt] \# V(T)=m}} \mu \times \nu( \R(T)). \]

\begin{theorem}
The $\RRR$-bilinear map $\widehat{c}$ is continuous.
\end{theorem}
\begin{proof}
Let $\mu _1,\mu_2 \in \SCurr (F_N)$ and $\mu_1^n, \mu_2^n\in \SCurr (F_N)\ (n=1,2,\dots )$ such that $\mu_i^n \rightarrow \mu_i\ (n\rightarrow \infty )\ (i=1,2)$. We prove that $\widehat{c}(\mu_1^n,\mu_2^n)\rightarrow \widehat{c}(\mu_1,\mu_2)\ (n\rightarrow \infty)$. 

Fix $\varepsilon >0$. 
Since $V\: \SCurr (F_N)\rightarrow \RRR $ is continuous, $V(\mu_i^n)\rightarrow V(\mu_i)\ (n\rightarrow \infty )$, and so we set 
\[ M:= \frac{3}{\varepsilon} \tn{sup}\{ V(\mu_1^n) V(\mu_2^n)\ |\ n=1,2,\dots \} (<\infty ).\]
Take a positive integer $L\geq M$. Then,
\begin{align*}
	&L\sum_{m=L}^{\infty } \sum_{\substack{[T]\in \SSub (X,\id )/F_N\\[1pt] \# V(T)=m}} \mu_1^n\times \mu_2^n( \R(T)) \\
\leq &\sum_{m=L}^{\infty } \sum_{\substack{[T]\in \SSub (X,\id )/F_N\\[1pt] \# V(T)=m}} m \mu_1^n\times \mu_2^n( \R(T)) \\
\leq &\sum_{m=1}^{\infty } \sum_{\substack{[T]\in \SSub (X,\id )/F_N\\[1pt] \# V(T)=m}} m \mu_1^n\times \mu_2^n( \R(T)) \\
= &\sum_{T\in \SSub (X,\id )} \# V(T) \mu_1^n\times \mu_2^n( \R(T)) \\
\leq  &V(\mu_1^n ) V(\mu_2^n).
\end{align*}
Consequently, we have
\begin{align*}
	&\sum_{m=L}^{\infty } \sum_{\substack{[T]\in \SSub (X,\id )/F_N\\[1pt] \# V(T)=m}} \mu_1^n\times \mu_2^n( \R(T)) \\
	\leq &\frac{1}{L}V(\mu_1^n ) V(\mu_2^n ) 
	\leq \frac{1}{M}V(\mu_1^n ) V(\mu_2^n ) 
	\leq \frac{\varepsilon}{3}\quad  (n=1,2\dots ).
\end{align*}
In the same way, we have 
\[ \sum_{m=L}^{\infty } \sum_{\substack{[T]\in \SSub (X,\id )/F_N\\[1pt] \# V(T)=m}} \mu_1\times \mu_2( \R(T)) \leq \frac{\varepsilon}{3}.\]

Since 
\[ \sum_{m=1}^{L-1 } \sum_{\substack{[T]\in \SSub (X,\id )/F_N\\[1pt] \# V(T)=m}} \mu_1^n\times \mu_2^n( \R(T)) \]
is a finite sum, this converges to 
\[ \sum_{m=1}^{L-1 } \sum_{\substack{[T]\in \SSub (X,\id )/F_N\\[1pt] \# V(T)=m}} \mu_1\times \mu_2( \R(T)) \]
when $n\rightarrow \infty$. If $n$ is large enough, then the absolute value of the difference of the above two sums is smaller than $\varepsilon /3$. Hence,
\[ |\widehat{c}(\mu_1^n,\mu_2^n )-\widehat{c}(\mu_1,\mu_2) |<\frac{\varepsilon}{3}+\frac{\varepsilon}{3}+\frac{\varepsilon}{3}=\varepsilon. \]
This completes the proof.
\end{proof}

\begin{proof}[Proof of Theorem \ref{thm:Neumann product}]
We define the $\RRR$-bilinear functional
\[ \N \:  \SCurr (F_N)\times \SCurr (F_N)\rightarrow \RRR \]
as
\[ \N (\mu ,\nu ):= \widehat{E}(\mu ,\nu )-\widehat{V}(\mu, \nu )+\widehat{c}(\mu ,\nu )\]
for $\mu ,\nu \in \SCurr (F_N)$. Then $\N$ is a continuous symmetric $\RRR$-bilinear functional, and we have $\N (\eta_{H},\eta_{K})=\N (H, K)$ for $ H,K\in \Sub (F_N)$. The uniqueness and $\tn{Out}(F_N)$-invariance of $\N$ follows immediately from the denseness of the rational subset currents in $\SCurr (F_N)$. 
\end{proof}

Now we recall the inequality in the SHNC (Theorem \ref{thm:Mineyev}).
Since the rational subset currents are dense in $\SCurr (F_N)$, and 
\[ \rk \times \rk :\SCurr (F_N)\times \SCurr (F_N)\to \RRR ;\ (\mu ,\nu )\mapsto \rk (\mu )\rk (\nu ) \]
is a continuous $\RRR$-bilinear map, we have the following corollary.
\begin{corollary}\label{cor:SHNC}
Let $\mu ,\nu \in \SCurr (F_N)$. The following inequality holds:
\[ \N (\mu ,\nu )\leq \rk (\mu )\rk (\nu ).\]
\end{corollary}

\section{The intersection map and the intersection functional}\label{sec:intersection}
Let $I$ be the \ti{intersection map}:
\[ I \: \C_N \times \C_N \rightarrow \{ \tn{closed subsets of } \partial X\} ;\ (S_1,S_2)\mapsto S_1\cap S_2. \]
For $\mu ,\nu \in \SCurr (F_N)$, the intersection map $I$ induces a Borel measure 
$I _{*}(\mu \times \nu )$ on $\C_N$ by pushing forward: for a Borel subset $U\subset \C_N$, we define
\[ I _{*}(\mu \times \nu ) (U):= \mu \times \nu (I ^{-1}(U)).\]
Note that it is not trivial that $I^{-1}(U)$ is a measurable set of $\mu \times \nu$ (see Appendix \ref{app:1}). The Borel measure $I _{*}(\mu \times \nu )$ is a subset current on $F_N$, that is, an $F_N$-invariant locally finite Borel measure on $\C_N$. When we consider the diagonal action of $F_N$ on $\C_N\times \C_N$, the product measure $\mu\times \nu$ is $F_N$-invariant. Therefore, $I _{*}(\mu \times \nu )$ is also an $F_N$-invariant Borel measure. 

Next, we check the local finiteness of $I _{*}(\mu \times \nu )$. For every $T\in \Sub (X,\id )$ 
\begin{align*}
I^{-1} (\SCyl (T))
&\subset \{ (S_1,S_2)\in \C \times \C \ |\ \mathrm{Conv}(S_1),\mathrm{Conv}(S_2)\ni \id \} \\
	&=\bigsqcup_{T_1,T_2\in \R _1}\SCyl (T_1)\times \SCyl (T_2).
\end{align*}
Thus we have
\begin{align*}
I _{*}(\mu \times \nu ) (\SCyl (T))&=\mu \times \nu (I^{-1}(\SCyl (T)))\\
	&\leq \mu \times \nu \left(\bigsqcup_{T_1,T_2\in \R _1}\SCyl (T_1)\times \SCyl (T_2)\right)\\
	&=V(\mu )V(\nu )<\infty.
\end{align*}
This implies that $I _{*}(\mu \times \nu ) $ is locally finite.

We define the $\RRR$-bilinear map
\[ \widehat{I}\: \SCurr (F_N)\times \SCurr (F_N)\rightarrow \SCurr (F_N),\]
as 
\[ \widehat{I}(\mu ,\nu ):= I_{*}(\mu \times \nu ). \]

For $H,K\in \Sub (F_N)$ and any Borel subset $U\subset \C_N$, we have
\begin{align*}
&\widehat{I}(\eta_H ,\eta _K )(U)\\
	=&\eta_H\times \eta_K (I^{-1}(U))\\
	=&\sum_{(g_1H,g_2K)\in F_N/H\times F_N/K} \gd_{g_1\Lambda (H)}\times \gd_{g_2\Lambda (K)}(I^{-1}(U))\\
	=&\sum_{(g_1H,g_2K)\in F_N/H\times F_N/K} \gd_{g_1\Lambda (H)\cap g_2\Lambda (K)}(U).
\end{align*}
Therefore, we have the following explicit description of $\widehat{I}(\eta_H,\eta_K)$:
\[ \widehat{I}(\eta_H ,\eta _K )=\sum_{(g_1H,g_2K)\in F_N/H\times F_N/K} \gd_{g_1\Lambda (H)\cap g_2\Lambda (K)}.\]
Since $\Lambda (H'\cap K')=\Lambda (H')\cap \Lambda (K')$ for any $H',K'\in \Sub (F_N)$,
we have
\[ \widehat{I}(\eta_H ,\eta _K )=\sum_{(g_1H,g_2K)\in F_N/H\times F_N/K} \gd_{\Lambda (g_1Hg_1^{-1}\cap g_2Kg_2^{-1})}.\]

\begin{theorem}\label{thm:int}
For $H,K\in \Sub (F_N)$ we have
\[ \widehat{I}(\eta _H,\eta _K)=\sum_{HgK\in H\backslash F_N/K}\eta_{H\cap gKg^{-1}}.\]
Hence, 
\[ \rk \circ \widehat{I}(\eta _H,\eta _K)=\N (\eta _H,\eta _K)\ (H,K\in \Sub (F_N)).\]
\end{theorem}
\begin{proof}
Let $H,K\in \Sub (F_N)$. We denote by $F_N\backslash ( F_N/H\times F_N/K) $ the quotient set of $F_N/H\times F_N/K$ by the natural diagonal action of $F_N$. Then, we have the bijective map:
\[ F_N\backslash ( F_N/H\times F_N/K)\to H\backslash F_N /K;\ [g_1H,g_2K]\mapsto Hg_1^{-1}g_2K.\]
Denote $H^g$ to be $gHg^{-1}$ for $g\in F_N$. Then for each $[g_1H,g_2K]\in F_N\backslash ( F_N/H\times F_N/K)$ with fixed $g_1$ and $g_2$ we have a bijective map:
\[ [g_1H,g_2K]\to F_N/(H^{g_1}\cap K^{g_2});\ (gg_1H,gg_2K)\mapsto g(H^{g_1}\cap K^{g_2}).\]
For $(g_1'H,g_2'K)\in [g_1H,g_2K]$ we can see that $H^{g_1'}\cap K^{g_2'}$ is conjugate to $H^{g_1}\cap K^{g_2}$. Therefore $\eta_{H^{g_1}\cap K^{g_2}}$ does not depend on the choice of $g_1$ and $g_2$. Consequently,
\begin{align*}
&\widehat{I}(\eta_H ,\eta _K ) \\
	=&\sum_{(g_1H,g_2K)\in F_N/H\times F_N/K} \gd_{\Lambda (H^{g_1}\cap K^{g_2})} \\
	=&\sum_{[g_1H,g_2K]\in F_N\backslash (F_N/H\times F_N/K)} \sum_{(gg_1H,gg_2K)\in [g_1H,g_2K]}\gd_{ g\Lambda
		(H^{g_1}\cap K^{g_2})}\\
	=&\sum_{[g_1H,g_2K]\in F_N\backslash (F_N/H\times F_N/K)} \sum_{g(H^{g_1}\cap K^{g_2})\in
	F_N/(H^{g_1}\cap K^{g_2})}\gd_{ g\Lambda ( H^{g_1}\cap K^{g_2})}\\
	=&\sum_{[g_1H,g_2K]\in F_N\backslash (F_N/H\times F_N/K)} \eta_{ H^{g_1}\cap K^{g_2}}\\
	=&\sum_{HgK\in H\backslash F_N /K} \eta_{H\cap gKg^{-1}},
\end{align*}
as required.
\end{proof}

Note that from Proposition \ref{prop:not conti}, the map $\widehat{I}$ is not continuous.
However, we can establish the following theorem. From Theorem \ref{thm:int} and Theorem \ref{thm:rank int} we can think of $\widehat{I}$ as a generalization of the construction of the fiber product graph, which we considered in the beginning of Section \ref{sec:product}. One of the points of $\widehat{I}$ is that we do not use the Cayley graph $X$ in the definition of $\widehat{I}$.

\begin{theorem}\label{thm:rank int}
The following equality holds:
\[ \rk \circ \widehat{I} =\N.\]
\end{theorem}
\begin{proof}
Let $\mu ,\nu \in \SCurr (F_N)$. We prove the above equality by representing $\rk\circ \widehat{I}(\mu, \nu)$ and $\N$ explicitly. Most parts of this proof consist of technical calculations.

First, by the definition of $\rk$ and $\widehat{I}$,
\begin{align*}
\rk \circ \widehat{I} (\mu, \nu) &=E(\hat{I}(\mu ,\nu ))-V(\hat{I}(\mu ,\nu ))\\
&=\sum_{a\in A}\mu\times \nu \Bigl( I^{-1}(\SCyl (e_a))\Bigr) -\mu \times \nu \left( I^{-1}\left (
\bigsqcup_{T\in \R_1}\SCyl (T)\right) \right).
\end{align*}
For any $(S_1,S_2)\in I^{-1}(\SCyl (e_a))$, we have 
\[ e_a \subset \mathrm{Conv}(S_1\cap S_2)\subset \mathrm{Conv}(S_1)\cap \mathrm{Conv}(S_2),\]
that is, $e_a\underset{\mathrm{ext}}{\subset} \mathrm{Conv}(S_1)$ and $e_a\underset{\mathrm{ext}}{\subset} \mathrm{Conv}(S_2)$. Consequently, 
\[I^{-1}(\SCyl (e_a))\subset \SCyl (e_a)\times \SCyl (e_a).\]
Similarly, for any $(S_1,S_2)\in I^{-1} (\bigsqcup_{T\in \R_1} \SCyl (T) )$ we have $\id \in \mathrm{Conv}(S_i)\ (i=1,2)$, which implies that
\begin{align*}
I^{-1}\left (\bigsqcup_{T\in \R_1}\SCyl (T)\right) &\subset \{ (S_1,S_2)\in \C \times \C \ |\ \mathrm{Conv}(S_1),\mathrm{Conv}(S_2)\ni \id \} \\
&= \bigsqcup_{T_1,T_2\in \R_1}\SCyl (T_1)\times \SCyl (T_2).
\end{align*}
By using difference sets, we have the following description
\[ \rk \circ \widehat{I} (\mu,\nu )=A_1-A_2-A_3+A_4 \label{eq0},\]
where
\begin{align*}
	A_1&=\sum_{a\in A}\mu \times \nu (\SCyl (e_a)\times \SCyl (e_a)); \\
	A_2&=\sum_{T_1,T_2\in \R_1} \mu \times \nu (\SCyl (T_1)\times \SCyl (T_2));\\
	A_3&=\sum_{a\in A}\mu \times \nu \biggl(\SCyl (e_a)\times \SCyl (e_a)\setminus I^{-1}(\SCyl (e_a))\biggr) ;\\
	A_4&=\mu \times \nu \left( \bigsqcup_{T_1,T_2\in \R_1}\SCyl (T_1)\times \SCyl (T_2)\setminus I^{-1}\left (\bigsqcup_{T\in \R_1}\SCyl (T)\right) \right) .
\end{align*}
From Section \ref{sec:product}, $A_1=\widehat{E}(\mu,\nu )$ and $A_2=\widehat{V}(\mu ,\nu)$. Hence it suffices to show that $-A_3+A_4=\widehat{c}(\mu ,\nu )$.

In the following Step 1 and Step 2, we deform $A_3$ and $A_4$, and in Step 3 we prove that the equation $-A_3+A_4=\widehat{c}(\mu ,\nu )$ holds.
\smallskip

\noindent \textbf{Step 1:} First, we consider $A_3$. Let $(S_1, S_2)\in \SCyl (e_a)\times \SCyl (e_a)$. By the definition of subset cylinders, $(S_1,S_2)$ does not belong to $I^{-1}(\SCyl (e_a))$ if and only if either $S_1\cap S_2=\emptyset $, or $S_1\cap S_2\not= \emptyset $ and $S_1\cap S_2 \subset \tn{Cyl} (e_a) $ or $S_1\cap S_2 \subset \tn{Cyl} ((e_a)^{-1})$, where we endow $e_a$ with the orientation such that $o(e_a)=\id ,t(e_a)=a$, and $(e_a)^{-1}$ is the inverse of $e_a$. In the case that $S_1\cap S_2=\emptyset $, we have
\begin{align*}
&\Bigl(\SCyl (e_a)\times \SCyl (e_a)\setminus I^{-1}(\SCyl (e_a))\Bigr) \cap I^{-1}(\{ \emptyset \} )\\
=&\Bigl(\SCyl (e_a)\times \SCyl (e_a) \Bigr) \cap I^{-1}(\{ \emptyset \} )\\
=&\bigsqcup_{\substack{T\in \Sub (X,\id )\\[1pt]\ T\supset e_a}} \R (T).
\end{align*}
In the case that $S_1\cap S_2\not=\emptyset $, for $a\in A\cup A^{-1}$ we put
\[ U(a):= \Bigl(\SCyl (e_a)\times \SCyl (e_a)\setminus I^{-1}(\SCyl (e_a))\Bigr) \setminus I^{-1}(\{ \emptyset \} )\  \]
where $e_a \ (a\in A^{-1})$ is the oriented edge with origin $\id $ and terminal $a$ in $X$. We will use $U(a)\ (a\in A^{-1})$ in Step 3.
It follows that
\begin{align}
A_3
	=&\sum_{a\in A}\mu \times \nu \biggl(\SCyl (e_a)\times \SCyl (e_a)\setminus I^{-1}(\SCyl (e_a))\biggr) \nonumber \\
	=&\sum_{a\in A}\mu \times \nu \biggl( \Bigr( \SCyl (e_a)\times \SCyl (e_a)\setminus I^{-1}(\SCyl (e_a))\Bigr)
\cap I^{-1}(\{ \emptyset \} ) \biggr) \nonumber \\
	& +\sum_{a\in A}\mu \times \nu (U(a)) \nonumber \\
	=&\sum_{a\in A}\sum_{\substack{T\in \Sub (X,\id )\\[1pt]\ T\supset e_a}} \mu \times \nu (\R (T)) +\sum_{a\in A}\mu \times \nu (U(a))\nonumber \\
	=&\sum_{[T]\in \SSub (X ,\id )/F_N} \# E_{\tn{top}}(T)\mu \times \nu (\R (T)) +\sum_{a\in A}\mu \times \nu (U(a)). \label{eq5}
\end{align}

\noindent \textbf{Step 2:} Next, we consider $A_4$ in a similar way. First, we have
\begin{align*}
 & \left( \bigsqcup_{T_1,T_2\in \R_1} \SCyl (T_1)\times \SCyl (T_2)\setminus I^{-1}\left (\bigsqcup_{T\in \R_1}\SCyl (T)\right) \right) \cap I^{-1}(\{ \emptyset \} ) \\
=& \left( \bigsqcup_{T_1,T_2\in \R_1} \SCyl (T_1)\times \SCyl (T_2) \right) \cap I^{-1}(\{ \emptyset \} ) \\ 
=&\bigsqcup_{T\in \SSub (X,\id ) } \R (T).
\end{align*}
Put 
\[ U:=\left( \bigsqcup_{T_1,T_2\in \R_1} \SCyl (T_1)\times \SCyl (T_2)\setminus I^{-1}\left (\bigsqcup_{T\in \R_1}\SCyl (T)\right) \right) \setminus I^{-1}(\{ \emptyset \} ).\]
Then,
\begin{align}
A_4
	&=\mu \times \nu \left( \bigsqcup_{T\in \SSub (X,\id ) } \R (T) \right) +\mu \times \nu (U) \nonumber \\
	&=\sum_{[T]\in \SSub (X,\id )/F_N} \# V(T)\mu \times \nu (\R (T)) +\mu \times \nu (U). \label{eq6}
\end{align}

\noindent \textbf{Step 3:} From Eqs. (\ref{eq5}) and (\ref{eq6}), we obtain
\begin{align*}
 \rk \circ \widehat{I} (\mu,\nu ) 
	=&\widehat{E}(\mu ,\nu ) -\widehat {V}(\mu ,\nu ) \\
	&-\sum_{[T]\in \SSub (X ,\id )/F_N} \# E_{\tn{top}}(T)\mu \times \nu (\R (T)) -\sum_{a\in A}\mu \times \nu (U(a))\\
	&+\sum_{[T]\in \SSub (X,\id )/F_N} \# V(T)\mu \times \nu (\R (T)) +\mu \times \nu (U).
\end{align*}
Since for any $T\in \SSub (X,\id )$ we have $\# V(T)-\# E(T)=\chi (T)=1$, 
\begin{align*}
\rk \circ \widehat{I} (\mu ,\nu )
	=&\widehat{E}(\mu ,\nu ) -\widehat {V}(\mu ,\nu ) \\
	& +\sum_{[T]\in \SSub (X,\id )/F_N} \mu \times \nu (\R (T)) +\mu \times \nu (U)-\sum_{a\in A}\mu \times \nu (U(a))\\
	=&\widehat{E}(\mu ,\nu ) -\widehat {V}(\mu ,\nu ) \\
	& +\widehat{c}(\mu ,\nu ) +\mu \times \nu (U)-\sum_{a\in A}\mu \times \nu (U(a))\\
	=&\N (\mu ,\nu ) +\mu \times \nu (U)-\sum_{a\in A}\mu \times \nu (U(a)).
\end{align*}

Now, we show that $\mu \times \nu (U)=\sum_{a\in A}\mu \times \nu (U(a))$.
By the definition of $U(a)\ (a\in A)$, for any $(S_1,S_2)\in U(a)$ we have $S_1\cap S_2 \not=\emptyset $ and either $S_1\cap S_2\subset \tn{Cyl }(e_a)$ or $S_1\cap S_2 \subset \tn{Cyl} ((e_a)^{-1})$.
For $e\in E(X)$ we set
\[ \mathcal{H} (e):=\{ S\subset \partial X \ |\ S\ \tn{is closed and }S\subset \tn{Cyl} (e)\}.\]
Then
\[ U(a)=\biggl( U(a)\cap I^{-1}(\mathcal{H}(e_a)) \biggr) \sqcup \biggl( U(a)\cap I^{-1}(\mathcal{H}((e_a)^{-1})) \biggr) \]
and
\begin{align*}
a^{-1} U(a)=&a^{-1} \biggl( \Bigl(\SCyl (e_a)\times \SCyl (e_a)\setminus   I^{-1}(\SCyl (e_a))  \Bigr) \setminus I^{-1}(\{ \emptyset \} ) \biggr) \\
	=&\Bigl(\SCyl (e_{a^{-1}})\times \SCyl (e_{a^{-1}})\setminus I^{-1}(\SCyl (e_{a^{-1}}))  \Bigr) \setminus I^{-1}(\{ \emptyset \} ) \\
	=&U(a^{-1}).
\end{align*}
In addition, since for $a\in A$ we have $a^{-1}(e_a)^{-1}=e_{a^{-1}}$, the following equality holds
\[ a^{-1} I^{-1}\bigl(\mathcal{H} ((e_a)^{-1})\bigr) = I^{-1}\bigl( \mathcal{H} (e_{a^{-1}})\bigr), \]
and so
\[ a^{-1}\biggl( U(a)\cap I^{-1}(\mathcal{H}((e_a)^{-1})) \biggr) =U(a^{-1})\cap I^{-1}(\mathcal{H}(e_{a^{-1}})) . \]
Recall that $\mu \times \nu$ is $F_N$-invariant with respect to the diagonal action of $F_N$ on $\C_N \times \C_N$. Then we obtain
\begin{align*}
&\sum_{a\in A} \mu \times \nu (U(a))\\
=&\sum_{a\in A}\left\{ \mu \times \nu \biggl( U(a)\cap I^{-1}(\mathcal{H}(e_a)) \biggr) +\mu \times \nu \biggl( U(a)\cap I^{-1}(\mathcal{H}((e_a)^{-1})) \biggr) \right\} \\
=&\sum_{a\in A}\left\{ \mu \times \nu \biggl( U(a)\cap I^{-1}(\mathcal{H}(e_a)) \biggr) +\mu \times \nu \biggl( U(a^{-1})\cap I^{-1}(\mathcal{H}(e_{a^{-1}})) \biggr) \right\} \\
=&\sum_{a\in A\cup A^{-1}}\mu \times \nu \biggl( U(a)\cap I^{-1}(\mathcal{H}(e_a)) \biggr) \\
=&\mu \times \nu \left( \bigsqcup_{a\in A\cup A^{-1}} U(a)\cap I^{-1}(\mathcal{H}(e_a))  \right) ,
\end{align*}
where we note that if $a, a'\in A\cup A^{-1}$ and $a\not=a'$, then $\mathcal{H}(e_a)\cap \mathcal{H}(e_{a'})=\emptyset$.
Now it suffices to show that the following equality holds:
\[ \bigsqcup_{a\in A\cup A^{-1}} U(a)\cap I^{-1}(\mathcal{H}(e_a))=U. \]
The key claim for this equality is that for a given $S\in \C_N$, $\mathrm{Conv}(S)\not\ni \id $ if and only if there exists $a\in A\cup A^{-1}$ such that $S\subset \tn{Cyl}(e_a)$. 

Take $(S_1,S_2)\in U(a)\cap I^{-1}(\mathcal{H}(e_a))\ (a\in A\cup A^{-1})$. Since $\mathrm{Conv}(S_i)\supset e_a$, there exists $T_i\in \R_1$ such that $S_i\in \SCyl (T_i)\ (i=1,2)$. In addition, $S_1\cap S_2\not=\emptyset $ by the definition of $U(a)$. Also, $S_1\cap S_2\subset \tn{Cyl}(e_a)$ implies that $\mathrm{Conv}(S_1\cap S_2)\not\ni \id $, namely $S_1\cap S_2\not\in \bigsqcup_{T\in \R_1} \SCyl (T)$. Therefore $S_1\cap S_2 \in U$.

Take $(S_1,S_2)\in U$. Then $S_1\cap S_2\not=\emptyset $ and $S_1\cap S_2 \not\in \bigsqcup_{T\in \R_1} \SCyl (T)$, and so there exists $a\in A\cup A^{-1}$ such that $S_1\cap S_2\subset \tn{Cyl }(e_a)$. Since $\mathrm{Conv}(S_i)\ni \id$, we have $\mathrm{Conv}(S_i) \supset e_a$. Therefore we obtain $(S_1,S_2)\in U(a)\cap I^{-1}(\mathcal{H}(e_a))$.
\end{proof}

\appendix
\section{}\label{app:1}
The purpose of this appendix is to show that for any Borel subset $U\subset \C_N$ the preimage $I^{-1}(U)$ of the intersection map $I$ is a measurable set of a product measure of two Borel measures on $\C_N$. We prove this in a general setting.
 
The notation in this appendix is different from that in the main text of this paper.

Let $X$ be a compact metrizable space. Then we see that $X$ is second countable. Fix a countable basis $\{ U_n \} _{n=1}^{\infty} $ of $X$.
Let $\C$ be the set of all closed (compact) subsets of $X$. We provide $\C$ with the Vietoris topology, which has the sub-basis consisting of all sets of the forms
\[ [U]_{\subset}:=\{ K\in \C\ |\ K\subset U\} \]
and 
\[ [U]_{\not= \emptyset }:=\{ K\in \C\ |\ K\cap U\not= \emptyset \} , \]
where $U$ is an open subset of $X$. See \cite{Kec95} for details of the Vietoris topology.
Since $[\emptyset ]_{\subset}=\{ \emptyset \}$, $\{ \emptyset \}$ is an isolated point of $\C$. The topology of the subspace $\C \setminus \{  \emptyset  \}$ coincides with the topology induced by the Hausdorff metric when we give a distance on $X$ which is compatible with the topology of $X$. Moreover, we can see that $\C$ is compact. Therefore $\C\setminus \{ \emptyset \}$ is a compact metrizable space, which implies that $\C \setminus \{ \emptyset \}$ is second countable, and so is $\C$. 

Let $O_X$ be the set of all open subsets of $X$ and $\mathcal{O}$ the set of all open subsets of $\C$. 
Then the $\sigma$-algebra $\sigma (\mathcal{O})$ generated by $\mathcal{O}$ is the set of all Borel subsets of $\C$.

Now, we consider the intersection map
\[ I\: \C\times \C\to \C ;\ (K_1,K_2)\mapsto K_1\cap K_2.\]
The goal of this appendix is to prove the following proposition.
\begin{proposition}\label{prop:int}
The intersection map $I$ is a Borel map, which means that for any Borel subset $S\subset \C$ the preimage $I^{-1}(S)$ is a Borel subset of $\C \times \C$.
\end{proposition}
Note that a measurable set of a product measure of two Borel measures on $\C$ is an element of the $\sigma$-algebra generated by the set
\[ \sigma (\mathcal{O})\times \sigma (\mathcal{O}):=\{ U_1\times U_2 \ |\ U_1,U_2\in \sigma (\mathcal{O}) \}. \]
Since $\C$ is a second countable space, the $\sigma$-algebra generated by $\sigma (\mathcal{O})\times \sigma (\mathcal{O})$ coincides with the $\sigma$-algebra generated by the set of all open subsets of $\C \times \C$.

To prove the above proposition we prepare a ``good'' generating set of $\sigma (\mathcal{O})$ as a $\sigma$-algebra, and it suffices to show that $I^{-1}(U)$ is a Borel subset of $\C \times \C$ for $U$ belonging to the ``good'' generating set of $\sigma (\mathcal{O})$.
First, since $\C$ is a second countable space, the sub-basis
\[ \{ [U]_{\subset } \ |\ U\in O_X \} \cup \{ [U]_{\ne}\ |\ U\in O_X \} \]
is a generating set of $\sigma (\mathcal{O})$. 

From now on, we assume that the countable basis $\{ U_n \} _{n=1}^{\infty} $ of $X$ is closed under finite union, that is, $\bigcup_{i=1}^k V_i \in \{ U_n\ | \ n=1,2,\dots \}$ for any $V_1,\dots, V_k \in \{ U_n\ | \ n=1,2,\dots \}$.

\begin{lemma}\label{lem:int1}
The set $\{ [U]_{\ne}\ |\ U\in O_X \}$ is a generating set of $\sigma (\mathcal{O})$.
\end{lemma}
\begin{proof}
Take any $U\in O_X$. It suffices to show that $[U]_{\subset }$ belongs to the $\sigma$-algebra generated by the above set.
For $K\in \C$ we can see that $K$ belongs to $[U]_{\subset}$ if and only if there exists $U_n$ such that $U^c \subset U_n$ and $K\cap U_n=\emptyset $. The ``if '' part is obvious. We prove the ``only if'' part. For any $p\in U^c$ there exists $U_{n_p}$ such that $p\in U_{n_p}$ and $U_{n_p}\cap K=\emptyset $. Since $U^c$ is compact, there exist $p_1,\dots p_m \in U^c$ such that $U^c \subset \bigcup_{i=1}^m U_{n_{p_i}}$ and $(\bigcup_{i=1}^m U_{n_{p_i}})\cap K=\emptyset $. Hence there exists $U_n$ such that $U^c \subset U_n$ and $U_n\cap K=\emptyset $.
From the above equivalence, we obtain
\[
[U]_{\subset }=\bigcup_{U^c \subset U_n} \{ K\in \C\ |\ K\cap U_n =\emptyset \} \\=\bigcup_{U^c \subset U_n} \left( [U_n]_{\ne } \right) ^c,
\]
as required.
\end{proof}

For a compact subset $A\subset X$, set $[A]_{\ne}:= \{ K\in \C\ |\ K\cap A\not =\emptyset \}$.

\begin{lemma}
The set $\{ [A]_{\ne }\ |\ A\subset X: \tn{compact} \}$ is a generating set of $\sigma (\mathcal{O})$.
\end{lemma}
\begin{proof}
First, note that for any compact subset $A\subset X$, we have
\[ [A]_{\ne }=\bigcap_{A\subset U_n}[U_n]_{\ne }.\]
Here there is $U_n$ containing $A$ since $\{ U_n \}$ is closed under a finite union. 
Hence $[A]_{\ne }$ belongs to $\sigma (\mathcal{O})$.
For any $U\in O_X$ by taking a sequence of compact subsets $\{K_n\}$ of $X$ such that $U=\bigcup _{n=1}^\infty K_n$, we have
\[ [U]_{\ne }=\bigcup _{n=1}^\infty [K]_{\ne },\]
as required.
\end{proof}

\begin{proof}[Proof of Proposition \ref{prop:int}]
Take any compact subset $A$ of $X$. We show that the preimage $I^{-1}([A]_{\ne })$ belongs to the $\sigma$-algebra generated by $\sigma (\mathcal{O}) \times \sigma (\mathcal{O})$.

For each positive integer $n$, take $p_n\in U_n$. Since $X$ is a metrizable space, we can define a distance function $d$ on $X$ which is compatible with the topology of $X$. For $x\in X$ and $r\geq 0$ we set
\[ B(x,r):=\{ y\in X \ |\ d(y,x)\leq r\}.\]
Take $(K_1,K_2)\in \C\times \C$. We show that $(K_1,K_2)$ belongs to $I^{-1}([A]_{\ne })$ if and only if for any positive integer $k$ there exists $p_n$ such that 
\[ K_i\cap A \cap B(p_n, \frac{1}{k})\not= \emptyset \quad (i=1,2).\]

First, we prove the ``only if'' part. Since $K_1\cap K_2\cap A\ne $, take $p\in K_1\cap K_2\cap A$ and take a subsequence $\{ p_{j_n}\}$ of $\{ p_n \}$ such that $\{ p_{j_n}\}$ converges to $p$. Then for any positive integer $k$ there exists $p_{j_n}$ satisfying the above condition.

Next, we prove the ``if'' part by contradiction. Assume that $K_1\cap K_2\cap A=\emptyset$. Then there exists a positive integer $k$ such that 
\[ \frac{1}{k}< \frac{1}{2}d(K_1\cap A, K_2\cap A).\]
From the assumption there exists $p_n$ such that
\[ K_i\cap A \cap B(p_n, \frac{1}{k})\not= \emptyset \quad (i=1,2).\] 
Hence, we can take $a_i \in K_i \cap A$ such that $d(a_i ,p_n )\leq 1/k \ (i=1,2)$. Therefore,
\[ d(a_1,a_2)\leq \frac{2}{k}<d(K_1\cap A, K_2\cap A),\]
which leads to a contradiction.

From the above, we have
\[ I^{-1}([A]_{\ne})=\bigcap_{k=1}^\infty \left( \bigcup_{n=1}^{\infty} [ A\cap B(p_n, \frac{1}{k})]_{\ne}\times [ A\cap B(p_n, \frac{1}{k})]_{\ne} \right),\]
as required.
\end{proof}

\subsection*{Acknowledgements.} I am deeply grateful to Prof. Masahiko Kanai who offered continuing support and constant encouragement. I also owe a very important debt to Prof. Katsuhiko Matsuzaki who provided sincere encouragement throughout the production of this study. 

The author also would like to thank the referee for her/his careful reading of the manuscript and valuable comments.

\end{document}